\documentclass{article}

\usepackage{fullpage}
\usepackage{amsmath}
\usepackage{amssymb,amsthm}
\usepackage{hyperref}
\hypersetup{
    colorlinks=true,
    linkcolor=blue,
    urlcolor=blue,
}
\usepackage{graphicx}
\usepackage{enumitem}

\DeclareMathOperator*{\argmin}{arg\,min}
\newcommand{\RR}{\mathbb{R}}

\DeclareMathOperator{\dist}{dist}
\DeclareMathOperator{\Fix}{Fix}

\theoremstyle{plain}
\newtheorem{theorem}{Theorem}[section]
\newtheorem{proposition}[theorem]{Proposition}
\newtheorem{lemma}[theorem]{Lemma}
\newtheorem{corollary}[theorem]{Corollary}

\theoremstyle{definition}
\newtheorem{definition}[theorem]{Definition}
\newtheorem{assumption}[theorem]{Assumption}

\theoremstyle{remark}
\newtheorem{remark}[theorem]{Remark}

\title{New operator designs for Halpern iterations with explicit rates under H\"older error bounds}

\begin{document}
\maketitle

\begin{center}
\begin{tabular}{ccc}

\begin{tabular}{c}
Pablo Barros\\
School of Applied Mathematics, FGV\\
Praia de Botafogo, Rio de Janeiro, Brazil\\
{\tt pabloacbarros@gmail.com}
\end{tabular}
&
&
\begin{tabular}{c}
Vincent Guigues\\
School of Applied Mathematics, FGV\\
Praia de Botafogo, Rio de Janeiro, Brazil\\
{\tt vincent.guigues@fgv.br}
\end{tabular}\\
&&\\
\begin{tabular}{c}
Roger Behling\\
Department of Mathematics, UFSC\\
Blumenau, SC, Brazil\\
{\tt rogerbehling@gmail.com}\\
\end{tabular}
&
&
\begin{tabular}{c}
Luiz-Rafael Santos\\
Department of Mathematics, UFSC\\
Blumenau, SC, Brazil\\
{\tt l.r.santos@ufsc.br}
\end{tabular}
\end{tabular}
\end{center}

\begin{abstract}
We investigate the asymptotic behavior of Halpern-type iterations applied to quasi-nonexpansive operators arising in best approximation problems over the intersection of finitely many closed convex sets in~\(\mathbb{R}^n\).  
Assuming a local decrease condition for the underlying operator and standard requirements on the stepsizes \((\alpha_k) \subset (0,1]\), we first prove strong convergence of the Halpern sequence \(x_{k+1} = \alpha_k x_0 + (1-\alpha_k) T x_k\) to the best approximation point \(x^\star\) in the intersection set, that is, the metric projection of \(x_0\) onto that set.  
Under the additional assumption that the intersection satisfies a H\"older-type error bound with exponent \(\gamma \in (0,1]\), we then derive explicit convergence rates for both feasibility and norm error: the distance from \(x_k\) to the intersection set decays like \(\mathcal O(\alpha_k^{\gamma/(2-\gamma)})\), while the norm error \(\|x_k - x^\star\|\) decays like \(\mathcal O(\alpha_k^{\gamma/(4-2\gamma)})\).  
These results apply to most projection-type operators used in convex feasibility problems (including MAP, CRM/SCCRM, Cimmino and 3PM/A3PM) and extend classical convergence analyses of the Halpern-type iterations by providing explicit, geometry-dependent rates governed by H\"older-type error bounds. Our numerical experiments show that Halpern-type iterations combined with some of these projection-type operators are quicker than Dykstra's algorithm to find the projection of a point in an intersection of ellipsoids or in an intersection of polyhedra.
\end{abstract}

\maketitle

\section{Introduction}

Finding the projection of a point onto the intersection of convex sets, the Best Approximation Problem (BAP), is a cornerstone task in applied mathematics, underpinning applications ranging from quantitative finance \cite{higham2002computing} to medical physics \cite{sidky2008image}.  

Throughout, $\RR^n$ denotes the $n$-dimensional Euclidean space equipped with an inner product $\langle \cdot,\cdot\rangle$ and associated norm \(\| \cdot \|.\) Given closed convex sets \(U_1,\ldots,U_m \subset \RR^n\) with nonempty intersection
\[
S := \bigcap_{i=1}^m U_i \neq \varnothing,
\]
and an anchor point \(x_0 \in \RR^n\), the associated best-approximation problem is
\begin{equation}\label{BAP}
\min_{x \in S} \|x - x_0\|. \tag{BAP}
\end{equation}
The unique optimizer of~\eqref{BAP} is the metric projection
\[
P_S(x_0) := \argmin_{x \in S} \|x - x_0\|.
\]

Beyond its intrinsic geometric appeal, the BAP is tightly interwoven with modern statistics and machine learning.
Already in classical order-restricted inference, computing restricted least--squares estimators amounts precisely to projecting an unconstrained estimate onto an intersection of affine and shape--constraint sets, and Dykstra's algorithm was originally introduced in this context as a practical BAP solver in Hilbert space \cite{dykstra1983algorithm}. The subsequent Boyle--Dykstra framework formalized projection methods for intersections of closed convex sets in a way that unifies a wide array of constrained regression and order-restricted estimation procedures \cite{boyle1986method}. More recent developments in constrained and shape-constrained estimation --- such as nonconvex isotonic regression based on submodular optimization \cite{bach2017efficient} --- can likewise be viewed as instances of projecting onto intricate intersections of convex (or ``nearly convex'') constraint sets.

In machine learning, projection onto intersections of convex sets arises routinely as a core subroutine in large-scale optimization. Examples include sparse group feature selection, where structured sparsity constraints and group penalties are enforced via iterative projection or proximal steps \cite{xiang2020efficient}, and universal latent-space model fitting for large networks with edge covariates, where projected gradient methods onto low-rank and simplex-type constraints underpin scalable estimation algorithms \cite{ma2020universal}. From an algorithmic viewpoint, Dykstra-type projection schemes now sit at the crossroads of several mainstream optimization paradigms: Tibshirani shows that Dykstra's algorithm for the BAP is exactly equivalent, on the dual, to coordinate descent for a broad class of regularized regression problems, and closely related to ADMM \cite{tibshirani2017dykstras}. Thus, advances in BAP algorithms immediately translate into new insights and performance guarantees for widely used procedures in statistics and machine learning.

Against this backdrop, obtaining sharp, geometry-dependent convergence guarantees for anchored projection schemes solving the BAP --- such as the Halpern iteration studied here --- is not merely of abstract interest. It feeds directly into the analysis of constrained regression, structured estimation, and large-scale learning algorithms whose inner loops are, at their core, best-approximation problems over intersections of convex sets.

Classical projection algorithms for feasibility, such as the method of alternating projections, are designed to produce {some} point in the intersection \(S\), but they do not in general converge to the specific best-approximation point \(P_S(x_0)\).  
To enforce convergence to \(P_S(x_0)\), Halpern \cite{Halpern1967} introduced the anchored fixed-point iteration
\begin{equation}\label{eq:halpern}
x_{k+1} = \alpha_k x_0 + (1-\alpha_k)\,T x_k,
\end{equation}
where \(T:\RR^n\to\RR^n\) is an operator with fixed-point set \(\Fix(T)=S\), the stepsizes \((\alpha_k)_{k\ge1}\subset(0,1]\) satisfy \(\alpha_k \to 0\) and \(\sum_k \alpha_k = \infty\), and \(x_0\) plays the role of an anchor.  

The Halpern iteration is by now a fundamental scheme in fixed-point theory and in the analysis of nonexpansive mappings; see the monograph of Bauschke and Combettes~\cite{Bauschke2017} and the classical convergence results in~\cite{wittmann1992approximation,khatibzadehRanjbar2015halpern}.  
It has been extensively studied in connection with convex feasibility, monotone inclusions, and best-approximation problems; see, e.g.,~\cite{BauschkeBorwein1993,BauschkeBorwein1996,Combettes1996,bauschkeKoch2015projection, khatibzadeh2020unification, SAEJUNG20103431}.  
In these works, the driving operator \(T\) is typically nonexpansive (or strongly quasi-nonexpansive), and the main focus is on strong convergence of \((x_k)\) to a point in \(\Fix(T)\), often identified with the metric projection \(P_S(x_0)\).

In the present paper, we are interested in the regime where \(T\) is generated by projection-type methods for convex feasibility and best approximation.  
Representative examples include von Neumann's method of alternating projections (MAP)~\cite{vonneumann1950functional, halperin1962product, BauschkeBorwein1993,BauschkeBorwein1996}, circumcentered reflection methods (CRM and their variants)~\cite{crm1,crmprod,BauschkeOuyangWang2018,BehlingBelloCruzIusemSantos2024,succccrm,barros2025parallelizing}, and parallel, projection-based schemes such as Cimmino's method \cite{BauschkeBorwein1996, Cimmino1938} and the Parallel Polyhedral Projection Method (3PM) along with its inexact variant A3PM \cite{Barros3pm}.  
For such operators, strong convergence of the Halpern iteration to an element of \(\Fix(T)\) in the nonexpansive setting is well understood.  
By contrast, explicit convergence rates for the {norm error} \(\|x_k - P_S(x_0)\|\) remain largely unavailable, especially when \(T\) is produced by more sophisticated projection mechanisms (e.g., circumcentered reflections or parallel polyhedral projections) that fail to be globally nonexpansive.

Our approach is explicitly geometric.  
We analyze the Halpern iteration~\eqref{eq:halpern} driven by a broad class of projection-based operators with \(\Fix(T)=S\), and we derive quantitative convergence rates for the best-approximation error \(\|x_k - P_S(x_0)\|\).  
These rates are governed by a H\"older-type error bound for the underlying intersection, thereby revealing a precise link between the geometric regularity of the sets \(\{U_i\}\) and the quantitative behavior of the anchored iteration.

In recent years, Halpern-type (anchored) iterations have also gained prominence in optimization, where they appear as structurally close to optimal fixed-point and accelerated first-order schemes.  
Tran-Dinh~\cite{TranDinh2023HalpernNesterov} established an equivalence between Halpern's method and Nesterov's acceleration under cocoercivity assumptions, framing the Halpern update as a unifying template for regularized first-order methods.  
This perspective underscores the importance of quantitative convergence guarantees for Halpern-type iterations in structured settings where the geometry of the solution set plays a central role.  
Our work leverages this viewpoint by showing how H\"older-type error bounds on the feasibility problem govern the convergence behavior of projection-driven Halpern iterations, even when the underlying operator is only locally decreasing and not globally nonexpansive.

For clarity of exposition, we present the analysis for the Halpern iteration generated by a single operator \(T\).  
However, the proofs extend verbatim to the more general update
\[
x_{k+1} = \alpha_k x_0 + (1-\alpha_k) T_k x_k,
\]
provided that each \(T_k\) satisfies Assumption~\ref{ass:T} with the same constant \(c_0>0\).  
This uniformity ensures that the geometric decrease of the distance to \(S\) (Lemma~\ref{lem:contraction}) and all subsequent recursive estimates remain valid.  
To keep the presentation streamlined, we therefore restrict attention to the single-operator case.

\noindent\paragraph{Related work.}
The Halpern iteration was introduced in the late 1960s as a tool for approximating fixed points of nonexpansive mappings in Hilbert spaces~\cite{Halpern1967} and has since become a standard scheme in fixed-point theory and monotone operator splitting (see also~\cite{Bauschke2017} and references therein).  
Classical results establish strong convergence of the Halpern sequence to the metric projection onto the fixed-point set under mild assumptions on the stepsizes and on the underlying space; see, e.g., Wittmann~\cite{wittmann1992approximation}, subsequent work by Reich and others, and strongly quasi-nonexpansive (SQNE)-type extensions such as~\cite{khatibzadehRanjbar2015halpern}.  

More recently, several authors have broadened the theoretical understanding of Halpern-type schemes.  
Yu and Wang~\cite{YuWang2023Halpern} generalized Halpern's iteration by relaxing control conditions and derived explicit convergence guarantees for split feasibility problems.  
Lin and Xu~\cite{LinXu2022Holder} developed convergence-rate results for generalized averaged nonexpansive operators under H\"olderian residual continuity, showing that convergence speed is dictated by the deviation from firm nonexpansiveness and by the underlying error-bound exponent.  
These works strongly support the paradigm that geometric error-bound structure drives algorithmic performance.  
Our analysis follows and sharpens this trajectory by quantifying feasibility and norm-error rates for Halpern-type iterations driven by quasi-nonexpansive projection schemes under minimal local assumptions.

Connections between Halpern-type iterations and accelerated optimization methods have also been studied intensively.  
Besides the equivalence result of Tran-Dinh~\cite{TranDinh2023HalpernNesterov}, Lieder~\cite{lieder2021halpern} showed that, for a single nonexpansive operator in a Hilbert space and a suitable choice of stepsizes, the Halpern iteration achieves the sharp residual rate \(\mathcal O(1/k)\) for \(\|x_k - T x_k\|\); this rate is in fact optimal in general Hilbert spaces.  
However, the emphasis in this line of work is predominantly on {residual}-type quantities, i.e., bounds on \(\|x_k - T x_k\|\) or, in optimization terminology, on gradients and operator residuals.  
Quantitative statements on the {norm error} \(\|x_k - x^\star\|\) or on the distance to the fixed-point set \(\mathrm{dist}(x_k, \Fix(T))\) are much less developed, especially beyond the globally nonexpansive case.

In parallel, there is now a mature literature relating convergence rates of projection algorithms to geometric regularity via error bounds and regularity properties.  
For alternating projections and related methods, one finds, for example, the analyses in~\cite{BauschkeBorwein1996,bauschkeKoch2015projection,borweinLiYao2014cyclic,drusvyatskiyLiWolkowicz2017illposedSDP}, where linear or H\"older-type regularity assumptions on the intersection yield sublinear rates with exponents explicitly determined by the regularity parameters.  
These results tie together metric (sub)regularity, Kurdyka--{\L}ojasiewicz inequalities, and the behavior of projection algorithms in both convex and nonconvex settings.

The present work lies at the intersection of these two strands.  
We study the Halpern iteration driven by projection-type operators and seek explicit rates for the distance to the target set and for the norm error to the best-approximation point, expressed in terms of a H\"older error bound for the underlying feasibility problem.  
To the best of our knowledge, such quantitative norm-error estimates for Halpern-type schemes in a setting where the driving operator arises from projection algorithms and satisfies only a local decrease property have not been available before.

This work extends geometric error-bound ideas to the Halpern iteration and to a broad class of quasi-nonexpansive operators, yielding what appears to be the first explicit rates for the norm error \(\|x_k - P_S(x_0)\|\) under a H\"older-type error bound.  
Our main contributions can be summarized as follows:
\begin{itemize}
\item \textbf{Halpern iteration beyond global nonexpansiveness.}  
We develop a convergence framework for Halpern-type iterations driven by projection-based operators that are not globally nonexpansive.  
This covers modular schemes such as CRM, A3PM, and MAP, many of which violate classical assumptions like firm nonexpansiveness or global Lipschitz continuity.  
Our analysis requires only a localized decrease condition and a geometric regularity property, thereby generalizing fixed-point frameworks from~\cite{YuWang2023Halpern,LinXu2022Holder} to a substantially richer class of projection operators.

\item \textbf{Explicit norm-error rates under H\"older error bounds.}  
We provide the first explicit convergence rates for the norm error of the Halpern sequence under a H\"older-type error bound on the fixed-point set.  
While prior work (e.g.,~\cite{TranDinh2023HalpernNesterov}) established residual or operator-fixed-point proximity rates, our analysis connects the iteration directly to the best-approximation target \(P_S(x_0)\).  
In particular, under an error bound with exponent \(\gamma\), we obtain a decay rate of order \(\mathcal O(\alpha_k^{\gamma/(4-2\gamma)})\) for \(\|x_k - P_S(x_0)\|\), via a careful study of feasibility and norm-error sequences.
\end{itemize}

Overall, the paper connects the asymptotic regularity theory of Halpern's method with the geometric error-bound analysis of projection algorithms~\cite{BauschkeBorwein1996,bauschkeKoch2015projection,borweinLiYao2014cyclic,drusvyatskiyLiWolkowicz2017illposedSDP}, and provides quantitative norm-error guarantees for anchored projection schemes solving the best-approximation problem~\eqref{BAP}.

The outline of the paper is as follows. Section \ref{sec:prel}
contains important preliminary results for our analysis.
The convergence and complexity analysis of Halpern iterations for operators $T$
satisfying the local decrease Assumption \ref{ass:T} below is discussed
in Section \ref{sec:mainresults}.
In Section~\ref{sec:examples}, we show that this local decrease assumption is
satisfied for six operators: the method of alternating projections
(MAP, or cyclic projections), Cimmino's method, 3PM, A3PM, the SCCRM operator
from~\cite{succccrm}, and the CRM operator.
Finally, the numerical experiments in Section~\ref{sec:numexp} show that
Halpern iterations combined with these operators can find the closest point to
an intersection of ellipsoids or polyhedra faster than Dykstra's algorithm.

\section{Preliminaries}\label{sec:prel}

Throughout, we denote
\begin{equation}\label{def:distnotation}
    \mathfrak d(z) := \mathrm{dist}(z, S),
    \qquad
    \delta(z) := \max_{i=1,\ldots,m} \mathrm{dist}(z, U_i),
\end{equation}
the distance to the intersection and the maximal constraint violation, respectively.  
These quantities play complementary roles in the analysis of projection algorithms.  
The distance \(\mathfrak d(z)\) measures how far a point is from feasibility, while \(\delta(z)\) captures the largest single-constraint violation and is often more directly controlled by one projection step.  
In the terminology of variational analysis, relations between \(\mathfrak d\) and \(\delta\) are encoded by metric (sub)regularity or linear/H\"older regularity of the collection \(\{U_i\}\); see, for instance, the classical work of Bauschke and Borwein on regularity and alternating projections~\cite{BauschkeBorwein1993,BauschkeBorwein1996}, and the more recent error-bound analyses of Drusvyatskiy, Lewis and collaborators~\cite{borweinLiYao2014cyclic,drusvyatskiyLiWolkowicz2017illposedSDP}.  
Our standing H\"older error-bound assumption precisely quantifies how the violation \(\delta(z)\) dominates the distance \(\mathfrak d(z)\), and it will be the main geometric ingredient underlying all rate estimates.

We begin with a simple lemma showing that, under mild boundedness assumptions, the vanishing of the maximal violation $\delta(z_k)$ forces the distance $\mathfrak d(z_k)$ to vanish as well. This continuity link between $\delta$ and $\mathfrak d$ is standard in the error-bound literature (see, e.g., \cite{BauschkeBorwein1996,borweinLiYao2014cyclic}). We include the short proof for completeness, as it provides the critical bridge that allows us to deduce the convergence of the distance sequence from the decay of constraint violations throughout our subsequent analysis.  

\begin{lemma}\label{lem:delta-to-d}
Let $U_1,\ldots,U_m\subset\mathbb R^n$ be nonempty, closed sets, and let
\[
S:=\bigcap_{i=1}^m U_i \neq \varnothing,\qquad
\mathfrak d(z):=\mathrm{dist}(z,S),\qquad
\delta(z):=\max_{1\le i\le m}\mathrm{dist}(z,U_i).
\]
If $(z_k)\subset\mathbb R^n$ is bounded and $\delta(z_k)\to0$, then $\mathfrak d(z_k)\to0$.
\end{lemma}

\begin{proof}
Let $z_{n_k}$ be a subsequence with \[ \mathfrak d(z_{n_k})\to \limsup \, \mathfrak d(z_k). \] By boundedness, pick a cluster point $\tilde z$ of $(z_{n_k})$. Taking a further subsequence if needed, we can assume that $z_{n_k} \to \tilde z.$ Since each $U_i$ is closed, $\mathrm{dist}( \, \cdot \,,U_i)$ is continuous; hence so is $\delta$, and $\delta(z_{n_k})\to0$ gives $\delta(\tilde z)=0$. Thus $\tilde z\in U_i$ for all $i$, i.e., $\tilde z\in S$, so $\mathfrak d(\tilde z)=0$. Since $S$ is closed, $\mathfrak d$ is continuous, from which we deduce \[ \limsup \, \mathfrak d(z_k)=\lim \, \mathfrak d(z_{n_k})=\mathfrak d(\tilde z)=0. \]
This clearly implies $\mathfrak d(z_k)\to0$.
\end{proof}

\noindent
Throughout the paper we repeatedly use the fact that projections onto closed convex
sets satisfy a form of the Pythagorean inequality.  
For convenience we record it here.

\begin{proposition}[Pythagorean inequality for projections]\label{prop:pythag}
Let \(X\subset\RR^n\) be a closed convex set and let \(z\in\RR^n\).
Then for the projection \(P_X(z)\), one has
\begin{equation}\label{eq:pythag-proj}
\|z - x\|^2
\;\ge\;
\|z - P_X(z)\|^2 + \|P_X(z) - x\|^2,
\qquad \forall\, x\in X.
\end{equation}
\end{proposition}

\begin{proof}
See, e.g., \cite[Lemma 2.2.8]{nesterov2018lectures}.
\end{proof}

\noindent
The next assumption encodes a uniform geometric decrease of the distance to~$S$ in terms of the violation~$\delta(x)$.
It resembles firm quasi-nonexpansiveness and is satisfied by most projection-type operators used in convex feasibility.

\begin{assumption}\label{ass:T}
The operator \(T:\RR^n\to\RR^n\) satisfies the local decrease property
\begin{equation}\label{eq:assT}
\|Tx - s\|^2
\;\le\;
\|x - s\|^2
\;-\;
c_0\,\delta(x)^2,
\qquad
\forall\,s\in S,\ \forall\,x\in\RR^n,
\end{equation}
for some constant \(c_0>0\).
\end{assumption}

\begin{remark}\label{rem:assT-iterates}
In the analysis below, the decrease estimate \eqref{eq:assT} is only invoked at
points that actually occur in the algorithmic sequence, i.e., with \(x=x_k\).
Therefore, it is enough to verify \eqref{eq:assT} for those \(x\) that can be
generated by the method from the chosen initialization. For simplicity, we keep
the statement of Assumption~\ref{ass:T} in the form \(\forall\,x\in\RR^n\).
\end{remark}

\noindent
Assumption~\ref{ass:T} expresses a uniform descent of the distance to \(S\) in terms of the maximal constraint violation~\(\delta(x)\),
and will serve as a key tool in the subsequent analysis.
In Section~\ref{sec:examples}, we verify it for the projection-based operators considered in this paper, including the method of alternating projections (MAP)~\cite{BauschkeBorwein1993,BauschkeBorwein1996},
the centralized circumcentered reflection method (cCRM) and its variants~\cite{crm1,crmprod,BehlingBelloCruzIusemSantos2024,succccrm,barros2025parallelizing},
and the Parallel Polyhedral Projection Method (3PM) and its inexact form A3PM (see \cite{Barros3pm}).
Several of these operators --- most notably cCRM and 3PM/A3PM --- are neither nonexpansive nor globally Lipschitz, and therefore fall outside the scope of classical Halpern convergence results in the nonexpansive or strongly quasi-nonexpansive (SQNE) setting (e.g.,~\cite{wittmann1992approximation,khatibzadehRanjbar2015halpern}).

Assumption~\ref{ass:T} can be viewed as a localized, set-valued variant of strong quasi-nonexpansiveness: when \(S=\Fix(T)\) and \(\delta(x)\) is replaced by \(\mathrm{dist}(x,\Fix(T))\), inequalities of the form~\eqref{eq:assT} reduce to the standard SQNE condition studied in fixed-point theory~\cite{khatibzadehRanjbar2015halpern}.  
Our approach, however, relies only on the decrease estimate~\eqref{eq:assT} combined with the anchored Halpern structure and does {not} require demiclosedness of \(I-T\) at the origin --- a hypothesis typically invoked to obtain strong convergence of Halpern-type methods in the SQNE framework (see, e.g., the demiclosedness-based arguments in~\cite{wittmann1992approximation,khatibzadehRanjbar2015halpern}).

In this way, Assumption~\ref{ass:T} cleanly separates the geometric contribution of the sets \(\{U_i\}\) (through the error bound) from the algorithmic contribution of the operator \(T\), while covering projection mappings that lie entirely outside the classical nonexpansive/SQNE setting.

\begin{definition}[Fixed-point set]
For an operator \(T:\RR^n\to\RR^n\), we define by
\[
\Fix(T):=\{x\in\RR^n:\;Tx=x\}
\]
its fixed-point set.
\end{definition}

\noindent
Assumption~\ref{ass:T} directly implies that the fixed-point set of~$T$ coincides with the feasible region~$S$,
ensuring that the limit points of the iteration are feasible.

\begin{lemma}[Fixed points of $T$]\label{lem:FixTS}
Under Assumption~\ref{ass:T} we have \(\mathrm{Fix}(T)=S\).
\end{lemma}

\begin{proof}
We proceed in two steps.
\begin{enumerate}[label=(\roman*)]
    \item \(S\subseteq\mathrm{Fix}(T)\).
If \(x\in S\), then \(\delta(x)=0\).
Taking \(s=x\) in~\eqref{eq:assT} yields
\(\|Tx-x\|^2\le0\), hence \(Tx=x\).

\item \(\mathrm{Fix}(T)\subseteq S\).
If \(Tx=x\), then for any \(s\in S\),
\[
\|x-s\|^2=\|Tx-s\|^2\le\|x-s\|^2-c_0\delta(x)^2,
\]
which forces \(\delta(x)=0\), i.e., \(x\in S\).
\end{enumerate}

Combining (i) and (ii) gives \(\mathrm{Fix}(T)=S\).
\end{proof}

\medskip

\noindent
To quantify the geometric coupling of the sets~$U_i$, we employ the notion of H\"older regularity,
which generalizes linear regularity and provides a unified framework for sublinear convergence rates.

\begin{definition}[H\"older Regularity]\label{def:holder-regularity}
Let \(U_1,\dots,U_m\subset H\) be closed, convex subsets of a Hilbert space \(H\), and set
\(S := \bigcap_{i=1}^m U_i.\) For \(\gamma\in(0,1]\), the collection \(\{U_i\}_{i=1}^m\) is {\(\gamma\)--H\"older regular}
if for every compact set \(K\subset H\) there exists a constant \(c>0\) such that
\[
\mathrm{dist}(x,S)\ \le\ c\left(\max_{1\le i\le m}\mathrm{dist}(x,U_i)\right)^{\gamma}
\qquad\text{for all }x\in K.
\]
\end{definition}

\noindent Using the notation \eqref{def:distnotation}, this means that there exist constants \(c > 0\) and \(\gamma \in (0,1]\) such that
\begin{equation}\label{eq:holder-eb}
\mathfrak d(z) \le c\, \delta(z)^{\gamma}
\qquad \text{for all } z \text{ in a compact set.}
\end{equation}
The exponent \(\gamma\) captures the local regularity of the intersection:
the case \(\gamma = 1\) corresponds to the familiar {linear regularity} setting,
while smaller \(\gamma\) indicate weaker, sublinear coupling among the sets.

\noindent
Clearly, \(\gamma\)-H\"older regularity implies that \(S \neq \varnothing\).
This property generalizes linear regularity (recovered when \(\gamma = 1\))
and is equivalent to the H\"older error bound~\eqref{eq:holder-eb}.
Convergence of alternating projections under mere feasibility --- without any rate guarantees --- is classical; see, for instance, Bauschke and Borwein~\cite{BauschkeBorwein1996}.  
Linear convergence under linear regularity was established by Lewis, Luke, and Malick~\cite{LewisLukeMalick2009},  
while sublinear rates under H\"older regularity were proved by Luke~\cite{Luke2013},  
building on the arguments of~\cite{LewisLukeMalick2009} and the later refinements in Luke and Thao~\cite{LukeThao2018}.
\medskip

\noindent
Combining the H\"older-type error bound~\eqref{eq:holder-eb} with the local decrease
property~\eqref{eq:assT} yields the following fundamental inequality.

\begin{lemma}[H\"older-EB Induced Contraction]\label{lem:contraction}
Suppose that Assumption~\ref{ass:T} holds and that the sets $\{U_i\}_{i=1}^m$ satisfy 
the $\gamma$--H\"older error bound~\eqref{eq:holder-eb} on a compact neighborhood of $S$.
Then it holds that
\[
\mathfrak d(T x)
\;\le\;
\mathfrak d(x)\, \left(1 - \tau\, \mathfrak d(x)^{\lambda} \right),
\]
where 
\[
\tau = \frac{c_0}{2c^{2 \gamma^{-1}}} \quad \text{and} \quad \lambda = 2(\gamma^{-1}-1),
\]
for all \(x\) in a compact neighborhood of \(S\).
\end{lemma}

\begin{proof}
From 
\[
\|T x - s\|^2 \le \|x - s\|^2 -  c_0 \, \delta(x)^2
\]
and
\[
\mathfrak d(x) \le c\, \delta(x)^{\gamma},
\]
we get
\begin{align*}
    \mathfrak d(T x)^2 &\le \mathfrak d(x)^2 - c_0 \, \left(\frac{\mathfrak d(x)}{c} \right)^{2 \gamma^{-1}} 
    \\&=\mathfrak d(x)^2 \left(1 - c_0 \, \frac{\mathfrak d(x)^{2 \gamma^{-1}-2}}{c^{2 \gamma^{-1}}}\right)
    \\&=: \mathfrak d(x)^2 \left(1 - 2 \tau \, \mathfrak d(x)^\lambda\right).
\end{align*}
The result then follows by the inequality
\[
1 - 2 \tau \, \mathfrak d(x)^\lambda \le \left(1 - \tau \, \mathfrak d(x)^\lambda\right)^2.
\]
\end{proof}

A similar decay relation under H\"older regularity was used in \cite{drusvyatskiyLiWolkowicz2017illposedSDP}
to establish sublinear convergence rates for the {method of alternating projections}.
In contrast, the present analysis applies the same geometric mechanism
to Halpern-type iterations and to broader classes of quasi-nonexpansive operators,
thereby generalizing the rate behavior beyond the alternating-projection setting.

\begin{remark}
By Lemma~\ref{lem:contraction}, the sequence $(\mathfrak d(T^k x))_{k\ge0}$ satisfies the nonlinear recurrence analyzed in
\cite[Lemma~4.1]{borweinLiYao2014cyclic}. 
When $\gamma=1$ (corresponding to the linearly regular case), 
the recurrence becomes geometric, and the {distance sequence} $(\mathfrak d(T^k x))$ converges linearly to $0$.
For $\gamma<1$, \cite[Lemma~4.1]{borweinLiYao2014cyclic} yields the explicit bound
\[
\mathfrak d(T^k x)
\;\le\;
\bigl(\mathfrak d(x)^{-(2-2\gamma)/\gamma} + (2-2\gamma)\gamma^{-1}\,\tau\,k\bigr)^{-\gamma/(2-2\gamma)}
=\mathcal O\,\bigl(k^{-\gamma/(2-2\gamma)}\bigr).
\]
Moreover, since $(T^k x)_{k\ge0}$ is Fej\'er monotone with respect to $S$ by~\eqref{eq:assT},
the Fej\'er-tail estimate \cite[Proposition~5.4 (iv)]{Bauschke2017}
argued as, e.g., in \cite[Theorem~5.12]{Bauschke2017}
converts the above distance estimate into convergence of the iterates to a point in~$S$
with the same order of decay.
\end{remark}

\medskip

\begin{assumption}\label{ass:alphan} The stepsize sequence \((\alpha_k)_{k\ge1}\) satisfies \begin{enumerate}[label=(\roman*)] \item \( \alpha_k\in(0,1]\); \item \( \alpha_k\to 0; \) \item \(\displaystyle \limsup_{k\to\infty}\,\left(\frac{1}{\alpha_{k+1}}-\frac{1}{\alpha_k}\right)<2; \) \item \(\displaystyle \lim_{k\to\infty} \, \frac{\alpha_{k}}{\alpha_{k+1}} = 1.\) \end{enumerate} \end{assumption}
\smallskip
\noindent
These conditions are fulfilled, for instance, by the standard choice
\[
\alpha_k = k^{-a},
\qquad a\in(0,1].
\]
Assumption~\ref{ass:alphan} is in line with the classical stepsize requirements for Halpern-type iterations.  
In his seminal paper, Halpern~\cite{Halpern1967} already identified the two basic asymptotic conditions
\[
\alpha_k \to 0
\quad\text{and}\quad
\sum_{k=1}^\infty \alpha_k = +\infty
\]
as essentially necessary for strong convergence, even in very simple examples.  
Subsequent works on Halpern iterations for nonexpansive mappings in Hilbert spaces~\cite{lions1977approximation,wittmann1992approximation,reich1994approximating,xu2002iterative}
established strong convergence under various, progressively refined sufficient conditions on \((\alpha_k)\),
typically combining \(\alpha_k \to 0\), non-summability, and mild regularity of the stepsizes.
More recently, even adaptive rules have been analyzed in this context~\cite{he2024convergence}.

In our setting, the additional technical requirements \textup{(iii)}--\textup{(iv)} control the discrete variation of \((\alpha_k)\) and are tailored to the quantitative error analysis in Section~\ref{sec:mainresults}.  
In particular, \textup{(ii)} and \textup{(iii)} together imply the classical non-summability condition, as shown in the lemma below.  
Assumption~\ref{ass:alphan} thus subsumes the standard Halpern conditions while providing sufficient regularity to derive explicit convergence rates.

The next lemma collects several elementary but useful properties that follow from~Assumption \ref{ass:alphan}. These estimates ensure the stability of later recursive inequalities.

\begin{lemma}\label{lem:limpsupalpha}
    Suppose that \((\alpha_k)_{k\ge 1}\) is a sequence of real numbers satisfying Assumption \ref{ass:alphan}. Then, for every \(p\in(0,1]\),
    \begin{enumerate}[label=(\roman*)]
    \item it holds that \[
\sum_{k\ge1}\alpha_k=\infty;
\]
    \item we have
    \[
    \limsup_k \, \frac{\alpha_k^{p}-\alpha_{k+1}^{p}}{\alpha_k} \le 0;
    \]
    \item for all $\varepsilon > 0$ with
    \[
    \varepsilon < 2 - p \cdot \limsup_{k\to\infty}\,\left(\frac{1}{\alpha_{k+1}}-\frac{1}{\alpha_k}\right)
    \]
    there exists \(k_0\) such that
    \[
    (1-\alpha_k)^2 
  \frac{\alpha_k^{p}}{\alpha_{k+1}^{p}} \le 1 - \varepsilon \alpha_k
    \]
    for all \(k\ge k_0.\)
    \end{enumerate}
\end{lemma}

\begin{proof}
(i) With Assumption \ref{ass:alphan}~(iii), choose $K$ such that
\[
\frac{1}{\alpha_{k+1}}-\frac{1}{\alpha_k} < 2
\qquad\forall \ k\ge K.
\]
Summing from $K$ to $k-1$ yields
\[
\frac{1}{\alpha_k}
\;<\;
\frac{1}{\alpha_K}+2(k-K),
\]
hence
\[
\sum_{k\ge1}\alpha_k
\;\ge\;
\sum_{k\ge K}\frac{1}{\alpha_K^{-1}+2(k-K)}
\,=\,\infty,
\]
which proves (i).
\medskip

\par (ii) Fix \(p\in(0,1]\). Since \(t\mapsto t^p\) is concave on \((0,\infty)\), we obtain
\begin{equation}\label{concavprop}
\alpha_k^p-\alpha_{k+1}^p
\le p\,\alpha_{k+1}^{\,p-1}(\alpha_k-\alpha_{k+1}).
\end{equation}
Dividing by \(\alpha_k\) and using \(\frac{\alpha_k-\alpha_{k+1}}{\alpha_k} =\alpha_{k+1}\left(\frac{1}{\alpha_{k+1}}-\frac{1}{\alpha_k}\right),\) the concavity estimate yields
\begin{equation}\label{firstineqalpha}
\frac{\alpha_k^p-\alpha_{k+1}^p}{\alpha_k}
\le
p\,\alpha_{k+1}^{\,p}\left(\frac{1}{\alpha_{k+1}}-\frac{1}{\alpha_k}\right),
\end{equation}
which gives (ii) by Assumption \ref{ass:alphan}~(iii) and \(\alpha_{k+1}^{\,p}\to 0.\)

\medskip

For (iii), by concavity again,
\begin{equation}\label{step:limsupalpha1}
    \frac{\alpha_k^p}{\alpha_{k+1}^p}-1
=\frac{\alpha_k^p-\alpha_{k+1}^p}{\alpha_{k+1}^p}
\stackrel{\eqref{concavprop}}{\le}
p\,\frac{\alpha_k-\alpha_{k+1}}{\alpha_{k+1}}
= p\,\alpha_k\left(\frac{1}{\alpha_{k+1}}-\frac{1}{\alpha_k}\right).
\end{equation}
Let \(L:= \limsup \,\left(\frac{1}{\alpha_{k+1}}-\frac{1}{\alpha_k}\right) < 2.\) Given $\varepsilon \in (0,2-pL),$ take  \(\mu>L\) with \(\varepsilon < 2 - p \mu.\) Then from \eqref{step:limsupalpha1}, there exists \(k_0\) such that for all \(k\ge k_0\),
\[
\frac{\alpha_k^p}{\alpha_{k+1}^p}
\le 1 + p\mu\,\alpha_k.
\]
Therefore, for \(k\ge k_0\),
\begin{align*}
    (1-\alpha_k)^2\frac{\alpha_k^p}{\alpha_{k+1}^p}
&\le
(1-2\alpha_k+\alpha_k^2)(1+p\mu\alpha_k)
\\&=
1-(2-p\mu)\alpha_k
+
(1-2p\mu)\alpha_k^2
+
p\mu\alpha_k^3
\\&=
1-(2-p\mu)\alpha_k+\mathcal{O}(\alpha_k^2).
\end{align*}
Since \(2-p\mu>\varepsilon\), the \(\mathcal{O}(\alpha_k^2)\) term is eventually bounded above by
\((2-p\mu-\varepsilon)\alpha_k\). Hence, for all sufficiently large \(k\),
\[
(1-\alpha_k)^2\frac{\alpha_k^p}{\alpha_{k+1}^p}
\le
1-\varepsilon\alpha_k,
\]
proving (iii).
\end{proof}

We shall use the following elementary consequence of Xu's scalar sequence lemma
\cite[Lemma~2.5]{xu2002iterative}, stated in the form needed below.

\begin{lemma}\label{lem:limitcontraction}
Let $(d_k)$ and $(\sigma_k)$ be real sequences. Assume that $(\nu_k)$ satisfies $\nu_k \in [0,1]$ for all large $k$ and
\[
\sum_{k=1}^\infty \nu_k = \infty,
\]
and that
\[
d_{k+1} \le (1-\nu_k)d_{k} + \nu_k \sigma_k,
\qquad k\ge1.
\]
If $\displaystyle \limsup_{k\to\infty} \sigma_k \le 0$, then $\displaystyle \limsup_{k\to\infty} d_k \le 0$.
\end{lemma}

\begin{proof}
For all large \(k\) with \(\nu_k\in[0,1]\), by monotonicity and convexity of
\(r\mapsto r^+:=\max\{0,r\}\),
\[
d_{k+1}^+
\le
\bigl((1-\nu_k)d_k+\nu_k\sigma_k\bigr)^+
\le
(1-\nu_k)d_k^+ + \nu_k\sigma_k^+.
\]
Since \(\limsup_k\sigma_k\le0\), we have \(\sigma_k^+\to0\). Xu's scalar sequence lemma, applied with
\[
s_k=d_k^+,\qquad \alpha_k=\nu_k,\qquad \beta_k=\sigma_k^+,\qquad \gamma_k=0,
\]
gives \(d_k^+\to0\), hence \(\limsup_k d_k\le0\).
\end{proof}

The following technical lemma controls the asymptotic order of sequences governed by a nonlinear recursion. It constitutes the main ingredient in the proof of the convergence rate theorem.

\begin{lemma}\label{lem:orderrec}
    Suppose that \((\beta_k)_{k\ge 1}\) is a sequence of nonnegative numbers and \((\alpha_k)_{k\ge 1}\) satisfies Assumption \ref{ass:alphan}. If \(M > 0, \ \tau \in (0,1]\) and \(\lambda \ge 0\) are constants such that
    \begin{equation}\label{eq:orderrec}
    \beta_{k+1} \le \alpha_k M + \beta_k \left(1 - \tau\, \beta_k^{\lambda} \right) \qquad \forall \ k\ge 1,
    \end{equation}
    then 
    \[
    \limsup_k \, \frac{\beta_k}{\alpha_k^{1/(1+\lambda)}} \le \left(\frac{M}{\tau}\right)^{1/(1+\lambda)}.
    \]
    In particular, \(\beta_k = \mathcal O\left(\alpha_k^{1/(1+\lambda)}\right).\)
\end{lemma}

\begin{proof} Let us denote 
\[
p:=\frac{1}{1+\lambda}, \qquad r_k := \frac{\alpha_k^p}{\alpha_{k+1}^p}, \qquad z_k := \frac{\beta_k}{\alpha_{k-1}^p}, \qquad \ell = \left(\frac{M}{\tau}\right)^{1/(1+\lambda)}.
\]
It suffices to show that
\[
\limsup z_k \le \ell.
\]
Substituting \(\beta_k = \alpha_{k-1}^p z_{k}\) and $M = \tau \ell^{1+\lambda}$ into \eqref{eq:orderrec} we obtain
\begin{equation}\label{gs658}
    \alpha_k^p z_{k+1}
\;\le\;
\alpha_k \tau \ell^{1+\lambda}
+
\alpha_{k-1}^p z_k
-
\tau\alpha_{k-1}^{p(1+\lambda)} z_k^{1+\lambda}
\;=\;
\alpha_k \tau \ell^{1+\lambda}
+
\alpha_{k-1}^p z_k
-
\tau\alpha_{k-1} z_k^{1+\lambda}.
\end{equation}
Subtracting $\alpha_k^p z_k$ from both sides and dividing by $\tau\alpha_k$ yields
\begin{equation}\label{step:orderrec1}
        \frac{z_{k+1}-z_k}{\tau\alpha_k^{1-p}}
\;\le\;z_k \cdot \frac{\alpha_{k-1}^p-\alpha_k^p}{\tau\alpha_k}
+
\ell^{1+\lambda}
-
\frac{\alpha_{k-1}}{\alpha_k}\,z_k^{1+\lambda}.
\end{equation}
First, we show that $(z_k)$ is bounded. By contradiction, assume that there exists a subsequence $z_{n_k}$ such that \(z_{n_k+1} \ge z_{n_k}\) and \(z_{n_k+1} \to \infty.\) From \eqref{gs658}:
\begin{equation*}
        z_{{n_k}+1}
\;\le\;
\alpha_{n_k}^{1-p}\tau \ell^{1+\lambda}
+
\frac{\alpha_{{n_k}-1}^p}{\alpha_{n_k}^p} z_{n_k}
\;\le\;
\tau \ell^{1+\lambda}
+
\frac{\alpha_{{n_k}-1}^p}{\alpha_{n_k}^p} z_{n_k}.
\end{equation*}
Since $\alpha_{{n_k}-1}^p/\alpha_{n_k}^p \to 1$ by Assumption \ref{ass:alphan}~(iv), we must also have \(z_{n_k} \to \infty.\) But from \eqref{step:orderrec1} and $z_{{n_k}+1}\ge z_{n_k}$:
\begin{equation}\label{gcdgva}
        0 
\;\le\;
\ell^{1+\lambda}
-
\frac{\alpha_{{n_k}-1}}{\alpha_{n_k}}\,z_{n_k} \left(z_{n_k}^{\lambda} - \frac{\alpha_{{n_k}-1}^p-\alpha_{n_k}^p}{\tau\alpha_{{n_k}-1}}\right)
\end{equation}
where 
\[
\limsup_{k \to \infty} \, \frac{\alpha_{{n_k}-1}^p-\alpha_{n_k}^p}{\tau\alpha_{{n_k}-1}} \le 0 \qquad \text{and} \qquad \lim_{k \to \infty} \frac{\alpha_{{n_k}-1}}{\alpha_{n_k}} = 1
\]
by Lemma~\ref{lem:limpsupalpha} (ii) and Assumption \ref{ass:alphan}~(iv), so the right hand side of \eqref{gcdgva} diverges to $-\infty;$ a contradiction. Therefore, we have shown that $(z_k)$ is bounded.

\medskip

\noindent Next, from \eqref{step:orderrec1} and convexity of \(t\mapsto t^{1+\lambda}\) on \((0,\infty)\) we get
\begin{align*}
                    \frac{z_{k+1}-z_k}{\tau\alpha_k^{1-p}}
\;&\le\;z_k \cdot \frac{\alpha_{k-1}^p-\alpha_k^p}{\tau\alpha_k}
+
 \ell^{1+\lambda}-z_k^{1+\lambda} 
+ \left(1-
\frac{\alpha_{k-1}}{\alpha_k} \right)\,z_k^{1+\lambda}
\\&\le\;z_k \cdot \frac{\alpha_{k-1}^p-\alpha_k^p}{\tau\alpha_k}
+
(1+\lambda)\ell^{\lambda}(\ell-z_k)
+ z_k \cdot \left(1-
\frac{\alpha_{k-1}}{\alpha_k} \right)\,z_k^\lambda
\\&=:\; (1+\lambda)\ell^{\lambda}(\ell-z_k) + z_k\varepsilon_k,
\end{align*}
and hence
\[
z_{k+1}-z_k \le \tau\alpha_k^{1-p} \, (1+\lambda)\ell^{\lambda}(\ell-z_k) + \tau\alpha_k^{1-p}z_k \, \varepsilon_k.
\]
Denoting \(c:=\tau(1+\lambda)\ell^{\lambda}\) and \(\tilde \varepsilon_k := (\tau \, z_k \, \varepsilon_k)/c\), the last inequality reads
\[
    z_{k+1}-z_k \le c \, \alpha_k^{1-p}(\ell-z_k) + c \, \alpha_k^{1-p} \tilde \varepsilon_k,
\]
or
\[
    (z_{k+1}-\ell) \le (1-c \, \alpha_k^{1-p})(z_k-\ell) + c \, \alpha_k^{1-p} \tilde \varepsilon_k.
\]
Using Lemma~\ref{lem:limpsupalpha} (ii), Assumption \ref{ass:alphan}~(iv) and boundedness of $(z_k)$,
we easily deduce
\[
\limsup_{k \to \infty} \,\frac{\alpha_{k-1}^p-\alpha_k^p}{\tau\alpha_k} \le 0 \qquad \text{and} \qquad \lim_{k \to \infty} \left(1-
\frac{\alpha_{k-1}}{\alpha_k} \right)\,z_k^\lambda = 0,
\]
hence \(\limsup_k \tilde \varepsilon_k \le 0.\) Since \(\sum \, c \, \alpha_k^{1-p} \ge \sum \, c \, \alpha_k = \infty,\) by Lemma \ref{lem:limitcontraction} applied to
\[
\nu_k := c \, \alpha_k^{1-p}, \qquad d_k := z_k-\ell, \qquad \sigma_k := \tilde \varepsilon_k
\]
we obtain
\[
\limsup \, (z_k-\ell) \le 0,
\]
as needed.
\end{proof}

\begin{remark}
The bound established in Lemma \ref{lem:orderrec} is tight. Indeed, let \(\alpha_k\) be any non-increasing sequence and define \(\beta_k = (M/\tau)^{1/(1+\lambda)} \alpha_k^{1/(1+\lambda)}\). Then,
\[
\alpha_k M - \tau \beta_k^{1+\lambda} = \alpha_k M - \tau \left( \frac{M}{\tau} \alpha_k \right) = 0.
\]
Consequently, the right-hand side of \eqref{eq:orderrec} reduces to \(\beta_k\). Since \(\alpha_k\) is non-increasing, \(\beta_{k+1} \le \beta_k\) holds, satisfying the recursion condition. The resulting ratio
\[
\frac{\beta_k}{\alpha_k^{1/(1+\lambda)}} = \left(\frac{M}{\tau}\right)^{1/(1+\lambda)}
\]
matches the upper bound exactly for all \(k\).
\end{remark}

\section{Main Results}\label{sec:mainresults}

\subsection{Global convergence of Halpern's method}

The next lemma is an important device used to conclude that 
a nonnegative sequence converges to zero once a suitable averaged decrease is present.  
It is particularly useful because it does not assume monotonicity or summability of the decrease terms.

\begin{lemma}\label{lem:gamma-an-zero}
Let $(d_k)\subset[0,\infty)$, and let $(\nu_k)\subset[0,\infty)$ and $(\sigma_k)\subset[0,\infty)$ satisfy $\sigma_k\to0$.  
Assume that
\begin{equation}\label{eq:an-recursion}
d_{k+1}\;\le\;d_k+\sigma_k-\nu_k
\qquad(\forall \ k\ge1),
\end{equation}
and that the following implication holds:
\begin{equation}\label{eq:gammaimplication}
    \text{if } \nu_{n_k}\to0 \text{ for some subsequence } (n_k) 
\text{ then } d_{n_k}\to0.
\end{equation}
Then $d_k$ converges to $0$.
\end{lemma}

\begin{proof} 
We first prove that $\liminf \nu_k = 0$. From \eqref{eq:an-recursion} we have
\[
d_{k+1}
\;\le\;
d_k - \bigl(\nu_k - \sigma_k\bigr).
\]
Suppose, for contradiction, that $\liminf \nu_k = \underline \gamma > 0$. This implies from $\sigma_k\to0$:
\[
\liminf_{k\to\infty} \ (\nu_k - \sigma_k) = \underline \gamma.
\]
Hence, there exists $k_0$ such that for all $k\ge k_0$, we have 
$\nu_k - \sigma_k \ge \tfrac{1}{2} \underline \gamma$, and therefore
\[
d_{k+1} \;\le\; d_k - \tfrac{1}{2} \underline \gamma
\qquad (\forall \ k\ge k_0),
\]
so $d_k$ becomes negative for large $k$, contradicting $d_k\ge0$. 
Thus our assumption was false and we conclude that $\liminf \nu_k = 0$, which in turn implies $\liminf d_k = 0$ by \eqref{eq:gammaimplication}.

Now we establish $\limsup d_k=0$. Assume $\limsup d_k = \bar a > 0$.  
Since $\liminf d_k = 0$, the sequence oscillates, i.e., there exist infinitely many indices $j_k$ such that
\[
d_{j_k}\le \frac{\bar a}{2},\qquad d_{j_k+1}>\frac{\bar a}{2}.
\]
Applying \eqref{eq:an-recursion} at these indices yields
\[
\frac{\bar a}{2}
< d_{j_k+1}
\le d_{j_k}+\sigma_{j_k}-\nu_{j_k}
\le \frac{\bar a}{2}+\sigma_{j_k}-\nu_{j_k},
\]
so that $0<\sigma_{j_k}-\nu_{j_k}$, that is, \ $\nu_{j_k}<\sigma_{j_k}$.  
Since $\sigma_{j_k}\to0$, we obtain $\nu_{j_k}\to0$.  
By the hypothesis \eqref{eq:gammaimplication}, $d_{j_k}\to0$.

Finally, recursion \eqref{eq:an-recursion} gives
\[
\limsup_{k\to\infty} \ d_{j_k+1}
\le
\lim_{k\to\infty} \ \bigl(d_{j_k}+\sigma_{j_k}-\nu_{j_k}\bigr)
=0,
\]
contradicting $d_{j_k+1} > \bar a/2$.  
Hence $\bar a=0$ and $\limsup d_k=0$.

\medskip
Combining both steps, we conclude that $\liminf d_k = \limsup d_k = 0$, and therefore $d_k\to0$.
\end{proof}

\begin{remark}
    We note that Lemma \ref{lem:gamma-an-zero} shares a conceptual similarity with Lemma 2.3 of \cite{wang2026improvement}. In particular, condition \eqref{eq:gammaimplication} acts as the precise analogue of condition (a4) in \cite[Lemma 2.3]{wang2026improvement}. Both results circumvent the traditional assumption of summability for the descent terms ($\nu_k$ in our case, and $c_n$ in theirs). Instead, both lemmas utilize a similar implication mechanism where the convergence of a descent indicator to zero along a subsequence forces the desired asymptotic behavior of the sequence.
\end{remark}

\noindent
We now establish the basic qualitative result: under Assumption~\ref{ass:T}
and standard stepsize conditions, the Halpern sequence converges strongly to
the best approximation point in~\(S\). This theorem is the analogue of the
classical strong convergence results for Halpern iteration with nonexpansive
operators; see, e.g.,
\cite{wittmann1992approximation,khatibzadehRanjbar2015halpern}. Here, the
fixed-point property \(\Fix(T)=S\) is not postulated a priori but deduced from
the local decrease inequality~\eqref{eq:assT}, and the limit point is identified
explicitly as the metric projection \(P_S(x)\) of the anchor \(x\).

\begin{theorem}\label{thm:globalconv}
Let $U_i$ be finitely many closed convex sets in $\RR^n$ and let $S := \bigcap_i U_i \neq \varnothing$.  
Suppose $T : \RR^n \to \RR^n$ satisfies Assumption~\ref{ass:T}.  
Let $(\alpha_k)\subset[0,1]$ satisfy $\alpha_k\to0$ and $\sum_k \alpha_k=\infty$.  
Given $x\in\RR^n$, define the Halpern iteration
\[
x_{k+1} = \alpha_k x + (1-\alpha_k) T x_k, 
\qquad x_0=x.
\]
Then $x_k \to P_Sx$.
\end{theorem}

\begin{proof}
Let $s \in S$. Subtracting $s$ from the iteration rule,
\begin{equation}\label{step:globalconv5}
x_{k+1}-s = \alpha_k(x-s) + (1-\alpha_k)(T x_k - s).
\end{equation}
Using the convexity of $\|\cdot - s\|^2$,
\begin{equation}\label{step:globalconv1}
\|x_{k+1}-s\|^2
\;\le\;
\alpha_k\|x-s\|^2 + (1-\alpha_k)\|T x_k - s\|^2.
\end{equation}
Using \eqref{eq:assT}, inequality~\eqref{step:globalconv1} yields
\begin{equation}\label{step:globalconv2}
    \begin{aligned}
            \|x_{k+1}-s\|^2
&\le
\alpha_k\|x-s\|^2 + (1-\alpha_k)(\|x_k - s\|^2 - c_0\,\delta(x_k)^2)
\\&=(1-\alpha_k) \, \|x_k - s\|^2 + \alpha_k\|x-s\|^2 - c_0 (1-\alpha_k) \delta(x_k)^2
    \end{aligned}
\end{equation}
Note that this implies, in particular,
\[
\|x_{k+1}-s\|^2 \le (1-\alpha_k) \, \|x_k - s\|^2 + \alpha_k\|x-s\|^2 \le \max \big\{\|x_k - s\|^2, \|x-s\|^2 \big\},
\]
so that $(x_k)$ is bounded. Applying \eqref{step:globalconv2} at $s=P_Sx_k$ gives
\begin{equation}\label{step:globalconv3}
    \begin{aligned}
          \mathfrak d(x_{k+1})^2 &\le \|x_{k+1}-P_Sx_k\|^2\\
&\le (1-\alpha_k) \, \mathfrak d(x_{k})^2 + \alpha_k\|x-P_Sx_k\|^2 - c_0 (1-\alpha_k) \delta(x_k)^2.
    \end{aligned}
\end{equation}
Since $(x_k)$ is bounded, $(P_Sx_k)$ is also bounded, so we can define
\[
b := \sup_k \|x-P_Sx_k\|^2 < \infty.
\]
Inequality \eqref{step:globalconv3} then implies
\begin{equation}\label{step:globalconv4}
          \mathfrak d(x_{k+1})^2 \le \mathfrak d(x_{k})^2 + \alpha_k b - c_0 (1-\alpha_k) \delta(x_k)^2.
\end{equation}
By Lemma \ref{lem:delta-to-d}, since $(x_k)$ is bounded, if $\delta(x_{n_k})\to0$ for some subsequence $(n_k)$, then $\mathfrak d(x_{n_k}) \to 0$ as well. We can thus apply Lemma \ref{lem:gamma-an-zero} for 
\[
d_k = \mathfrak d(x_{k})^2, \qquad \nu_k = c_0 (1-\alpha_k) \delta(x_k)^2, \qquad \sigma_k = \alpha_k b.
\]
to get
\begin{equation}\label{step:globalconv6}
\mathfrak d(x_{k}) \to 0.
\end{equation}
Now let $s^\star=P_Sx.$ We next show that $\limsup_{k} \ \langle x_k - s^\star,\, x - s^\star\rangle \le 0.$ 
Suppose, by contradiction, that $\limsup_{k} \ \langle x_k - s^\star,\, x - s^\star\rangle> 0.$ Then there exist $\varepsilon>0$ and a subsequence $(n_k)$ such that 
\[
\langle x_{n_k}-s^\star,\,x-s^\star\rangle \ge \varepsilon
\qquad \forall \ k.
\]
Let $\tilde x$ be a cluster point of $(x_{n_k})$. 
Since $(x_k)$ is bounded, such a point exists, and by continuity of $\mathfrak d$ and \eqref{step:globalconv6} we have $\mathfrak d(\tilde x)=0$, i.e., $\tilde x\in S$.  
Passing to the limit along a convergent subsequence of $(x_{n_k})$, we obtain 
\[
\langle \tilde x - s^\star,\, x - s^\star\rangle \ge \varepsilon.
\]
However, since $s^\star=P_Sx$ and $\tilde x\in S$, the projection property implies 
\(\langle \tilde x - s^\star,\, x - s^\star\rangle \le 0,\)
a contradiction.  
Hence
\[
\limsup_{k\to\infty}\, \langle x_k - s^\star,\, x - s^\star\rangle \le 0.
\]
Using \eqref{step:globalconv5} and \eqref{eq:assT}, one verifies that
\[
\|(x_{k+1}-s^\star) - \alpha_k(x-s^\star)\|^2 = (1-\alpha_k)^2\|Tx_k-s^\star\|^2 \le (1-\alpha_k)^2\|x_k-s^\star\|^2.
\]
Expanding the square yields
\begin{align}
\|x_{k+1}-s^\star\|^2
& \le (1-\alpha_k)^2\|x_k-s^\star\|^2
 + 2\alpha_k\langle x_{k+1}-s^\star,\,x-s^\star\rangle -\alpha_k^2 \|x-s^\star\|^2 \\
& \le
(1-\alpha_k)\|x_k-s^\star\|^2
 + 2\alpha_k\langle x_{k+1}-s^\star,\,x-s^\star\rangle.
\end{align}
Applying Lemma~\ref{lem:limitcontraction} to
\[
\nu_k := \alpha_k \in [0,1], \qquad d_k := \|x_k - s^\star\|^2,
\qquad
\sigma_k := 2 \, \langle x_{k+1}-s^\star,\,x-s^\star\rangle,
\]
we conclude that $\|x_k - s^\star\|\to 0$, that is,
\[
x_k \to s^\star=P_Sx.
\]

\end{proof}

\noindent
In particular, taking \(T\) to be the composition of a finite number of metric projections onto closed convex sets recovers the usual convergence of Halpern-type projection methods for the convex feasibility problem, now under a unified assumption~\ref{ass:T}.  
The main purpose of the next section is to upgrade this qualitative result to {quantitative} rates that depend explicitly on the geometric regularity of \(\{U_i\}\).

\noindent
The above theorem guarantees global convergence without any regularity assumption;
the next section develops explicit convergence rates under the H\"older error bound.

\subsection{Rate of convergence under H\"older error-bound}

We now turn to the convergence properties of the Halpern iteration~\eqref{eq:halpern}
under the assumptions stated above.

\noindent
The next result quantifies the decay of the distance to the feasible set \(S\) along the Halpern iterates.  
It combines the geometric contraction from Lemma~\ref{lem:contraction} --- which exploits the local decrease property~\eqref{eq:assT} together with the H\"older error bound --- with the nonlinear recursion estimate of Lemma~\ref{lem:orderrec}.  
The resulting rate depends explicitly on the H\"older regularity exponent~\(\gamma\); when \(\gamma = 1\), it recovers a decay of the same order as the stepsizes \((\alpha_k)\), while for \(\gamma < 1\) it yields a genuinely sublinear behavior.  
This mirrors, in the setting of anchored Halpern iteration, the error-bound-based analysis of alternating projections and related projection schemes where the exponent in the error bound dictates the rate of convergence; see, for example,~\cite{borweinLiYao2014cyclic,drusvyatskiyLiWolkowicz2017illposedSDP}.

\begin{theorem}[Rate for the distance to \(S\)]\label{thm:rate-distance}
Let \((x_k)\) be generated by the Halpern iteration~\eqref{eq:halpern} with stepsizes satisfying~Assumption \ref{ass:alphan}.  
Suppose that \(T\) fulfills the local decrease property of Assumption~\ref{ass:T} with constant $c_0 \le 1$
and that the sets \(\{U_i\}\) satisfy the H\"older error bound~\eqref{eq:holder-eb} with exponent \(\gamma\in(0,1]\) and constant $c \ge 1$.
Then
\[
\limsup_k \ \frac{\mathfrak d(x_k)}{\alpha_k^{\,\gamma/(2-\gamma)}} \le \left(\frac{2c^{2 \gamma^{-1}} \ \mathfrak d(x)}{c_0} \right)^{\,\gamma/(2-\gamma)}.
\]
In particular, \(\mathfrak d(x_k)=\mathcal O\left(\alpha_k^{\,\gamma/(2-\gamma)}\right).\)
\end{theorem}

\begin{proof}
    Let $\lambda:=2(\gamma^{-1}-1)$ as in Lemma~\ref{lem:contraction}. From
    \[x_{k+1} = \alpha_k x + (1-\alpha_k) Tx_{k}\]
    and convexity of \(\mathfrak d(\, \cdot \,)\), we obtain
    \[
    \mathfrak d(x_{k+1}) \le \alpha_k \, \mathfrak d(x) + (1-\alpha_k) \, \mathfrak d(Tx_{k}).
    \]
    Now by Lemma \ref{lem:contraction},
    \begin{align*}
            \mathfrak d(x_{k+1}) &\le \alpha_k \, \mathfrak d(x) + (1-\alpha_k) \, \mathfrak d(x_k)\left(1 - \tau\, \mathfrak d(x_k)^{\lambda} \right) \\&\le \alpha_k \, \mathfrak d(x) + \mathfrak d(x_k)\left(1 - \tau\, \mathfrak d(x_k)^{\lambda} \right),
    \end{align*}
    where
    \[
    \tau := \frac{c_0}{2c^{2/\gamma}}.
    \]
    Since \(c_0\le1\) and \(c\ge1\), we have \(\tau\in(0,1]\). Then we are done by Lemma \ref{lem:orderrec} applied to 
    \[
    \beta_k = \mathfrak d(x_k), \quad M = \mathfrak d(x).
    \]
\end{proof}

The distance estimate in Theorem~\ref{thm:rate-distance} quantifies how fast the iterates approach the feasible set \(S\).  
For algorithmic purposes, however, the more relevant quantity is the actual norm error with respect to the best approximation point \(P_S(x)\).  
By combining the distance decay with a quantitative stability estimate for the metric projection, we obtain the following norm-convergence rate for the Halpern sequence, which, to the best of our knowledge, is the first explicit norm-error rate for Halpern iterations.

\begin{theorem}[Rate of convergence]\label{thm:rate-norm}
Under the hypotheses of Theorem~\ref{thm:rate-distance}, the Halpern iterates satisfy
\begin{align*}
        \limsup_k \ \frac{\|x_k - P_S(x)\|}{\alpha_k^{\,\gamma/(4-2\gamma)}} \le \frac{c_0^{\,-\gamma/(4-2\gamma)}}{\left(2 - \frac{1}{2 \gamma^{-1}-1} \cdot L \right)^{1/2}}\, [2 \, c \ \mathfrak d(x)]^{\,1/(2-\gamma)},
\end{align*}
where
\[
L := \displaystyle \limsup_{k\to\infty}\,\left(\frac{1}{\alpha_{k+1}}-\frac{1}{\alpha_k}\right).
\]
In particular, \(\|x_k - P_S(x)\| =\mathcal O\left(\alpha_k^{\,\gamma/(4-2\gamma)}\right)\).
\end{theorem}
\begin{proof}
Denote 
\[
s^\star := P_S(x), 
\qquad 
s_k := P_S(x_k), 
\qquad 
\lambda:=2(\gamma^{-1}-1), \qquad
p := \frac{1}{1+\lambda}.
\]
By the update rule,
\[
x_{k+1} - s^\star 
= 
\alpha_k (x - s^\star) + (1-\alpha_k)(T x_k - s^\star).
\]
Expanding the norm and using the quasi-nonexpansivity of \(T\), we obtain
\begin{align*}
\|(x_{k+1}-s^\star) - \alpha_k(x-s^\star)\|^2 
&\le (1-\alpha_k)^2 \|x_k - s^\star\|^2,
\end{align*}
which implies
\begin{align}
\|x_{k+1}-s^\star\|^2 
&\le (1-\alpha_k)^2 \|x_k - s^\star\|^2 
   + 2\alpha_k \langle x_{k+1}-s^\star,\, x-s^\star \rangle.\label{firstineqxs}
\end{align}
Since \(s_{k+1}\in S\), the projection property onto a convex set gives 
\[
\langle s_{k+1}-s^\star,\, x-s^\star\rangle \le 0.
\]
Therefore,
\begin{align}
\langle x_{k+1}-s^\star,\, x-s^\star\rangle
&=
\langle x_{k+1}-s_{k+1},\, x-s^\star\rangle
+
\langle s_{k+1}-s^\star,\, x-s^\star\rangle \nonumber\\
&\le
\langle x_{k+1}-s_{k+1},\, x-s^\star\rangle \nonumber\\
&\le
\|x_{k+1}-s_{k+1}\|\,\|x-s^\star\|. \label{csineqok}
\end{align}
Substituting \eqref{csineqok} into \eqref{firstineqxs} yields
\begin{equation}\label{step:mainbound}
\|x_{k+1}-s^\star\|^2
\le
(1-\alpha_k)^2 \|x_k - s^\star\|^2
+ 2\alpha_k\,\|x_{k+1}-s_{k+1}\|\,\|x-s^\star\|.
\end{equation}

\smallskip

\noindent
Define the sequences
\[
G_k = \frac{\|x_{k}-s_{k}\|}{\alpha_{k}^{p}}, \qquad z_k := \frac{\|x_k - s^\star\|^2}{\alpha_k^{p}}.
\]
Using this notation and \eqref{step:mainbound} gives
\[
\|x_{k+1}-s^\star\|^2 
\le 
(1-\alpha_k)^2 \|x_k - s^\star\|^2 
+ 2G_{k+1} \, \alpha_k \, \alpha_{k+1}^{p} \, \|x-s^\star\|. 
\]
Dividing both sides by \(\alpha_{k+1}^{p}\) yields
\begin{equation}\label{eq:znrec}
\frac{\|x_{k+1}-s^\star\|^2}{\alpha_{k+1}^{p}}
\le 
(1-\alpha_k)^2 
  \frac{\alpha_k^{p}}{\alpha_{k+1}^{p}}
  \frac{\|x_k - s^\star\|^2}{\alpha_k^{p}}
+ 
2G_{k+1} \, \alpha_k \|x-s^\star\|.
\end{equation}
Then \eqref{eq:znrec} can be rewritten as
\begin{equation}\label{step:rate-norm1}
    z_{k+1}
\le 
(1-\alpha_k)^2 
  \frac{\alpha_k^{p}}{\alpha_{k+1}^{p}} \,
z_k 
+ 
\alpha_k
\left(
2G_{k+1} \, \|x-s^\star\|\right).
\end{equation}

\smallskip

\noindent
Let 
\[
u := 2 - p \cdot L.
\]
By Lemma \ref{lem:limpsupalpha}~(iii), given \(\varepsilon \in (0, u)\), there exists \(k_0\) such that for all \(k\ge k_0\),
\[
(1-\alpha_k)^2 
  \frac{\alpha_k^{p}}{\alpha_{k+1}^{p}} \le 1-\varepsilon \alpha_k.
\]
Additionally, choose any $G < \infty$ with 
\[
G > \left(\frac{2c^{2 \gamma^{-1}} \ \mathfrak d(x)}{c_0} \right)^{\,1/(1+\lambda)}.
\]
Because of Theorem \ref{thm:rate-distance}, there is some $n_1$ such that $G_k \le G$ for all $k\ge n_1.$ Consequently, for all \(k\ge \max\{k_0, n_1\}\), inequality \eqref{step:rate-norm1} leads to
\begin{align*}
    z_{k+1} 
& \le 
(1-\varepsilon \alpha_k) z_k + \alpha_k 2G \, \|x-s^\star\|
\\& = (1-\varepsilon \alpha_k) z_k + \varepsilon \alpha_k \left(\frac{2G \, \|x-s^\star\|}{\varepsilon} \right),
\end{align*}
which can be rewritten as
\[
\left(z_{k+1} - \frac{2G \, \|x-s^\star\|}{\varepsilon} \right) \le (1-\varepsilon \alpha_k) \left(z_{k} - \frac{2G \, \|x-s^\star\|}{\varepsilon} \right)
\]
Since \(\sum_{k=1}^{+\infty} \varepsilon \alpha_k = \infty,\) Lemma \ref{lem:limitcontraction} applied to
\[
\nu_k = \varepsilon \alpha_k, \qquad d_k = z_{k} - \frac{2G \, \|x-s^\star\|}{\varepsilon}, \qquad \sigma_k = 0
\]
gives
\[
\limsup \, z_{k} - \frac{2G \, \|x-s^\star\|}{\varepsilon} \le 0 \implies \limsup \, z_{k} \le \frac{2G \, \|x-s^\star\|}{\varepsilon}.
\]
Since 
\[
\varepsilon \in (0, u) \qquad \text{and} \qquad G > \left(\frac{2c^{2 \gamma^{-1}} \ \mathfrak d(x)}{c_0} \right)^{\,1/(1+\lambda)}
\] 
where $u$ and $G$ are arbitrary, we conclude that
\[
\limsup \, z_{k} \, \le \, \frac{2 \|x-s^\star\|}{u} \, \left(\frac{2c^{2 \gamma^{-1}} \ \mathfrak d(x)}{c_0} \right)^{\,1/(1+\lambda)}.
\]
From the definition of \((z_k)\), this can be written
\begin{align*}
    \limsup \ \frac{\|x_k - s^\star\|^2}{\alpha_k^{p}} &\le \frac{2 \, \mathfrak d(x)}{2 - p \cdot L} \, \left(\frac{2c^{2 \gamma^{-1}} \ \mathfrak d(x)}{c_0} \right)^{\,1/(1+\lambda)}
    \\&= \frac{1}{2 - p \cdot L} \, \left(\frac{c^{2 \gamma^{-1}}}{c_0} \right)^{\,\gamma/(2-\gamma)} \, [2 \, \mathfrak d(x)]^{\,2/(2-\gamma)}.
\end{align*}
Finally, we arrive at
\begin{align*}
    \limsup \ \frac{\|x_k - s^\star\|}{\alpha_k^{p/2}} \le \frac{c_0^{\,-\gamma/(4-2\gamma)}}{\left(2 - p \cdot L \right)^{1/2}}\, [2 \, c \ \mathfrak d(x)]^{\,1/(2-\gamma)},
\end{align*}
which ends the proof of the theorem.
\end{proof}

In the regime \(\gamma=1\) (linear regularity of the feasible set), the exponent in the bound reduces to \(1/2\), so that the norm error decays essentially like the square root of the distance bound in Theorem~\ref{thm:rate-distance}.  
For \(\gamma<1\), the rate becomes sublinear and interpolates smoothly between these two extremes.  
Thus the H\"older exponent provides a single geometric parameter that governs both the speed at which the iterates approach feasibility and the speed at which they approach the best approximation point.  
From a practical standpoint, this clarifies how the ill-conditioning of the intersection (encoded in \(\gamma\)) degrades the performance of Halpern-type projection methods.

\begin{corollary}[Rate under linear regularity / Slater]\label{cor:rate-norm-gamma1}
Under the hypotheses of Theorem~\ref{thm:rate-distance}, suppose in addition that the H\"older regularity exponent is $\gamma=1$. Then the Halpern iterates satisfy
\[
    \limsup_{k\to\infty} \ 
    \frac{\|x_k - P_S(x)\|}{\alpha_k^{1/2}}
    \;\le\;
    \frac{2c\,\mathfrak d(x)}{\sqrt{c_0}}\,
    \left(
        2
        -
        \limsup_{k\to\infty} \ 
        \left(\frac{1}{\alpha_{k+1}} - \frac{1}{\alpha_k}\right)
    \right)^{-1/2}.
\]
In particular,
\[
    \|x_k - P_S(x)\|
    = \mathcal O\bigl(\alpha_k^{1/2}\bigr).
\]
\end{corollary}

\smallskip
\noindent
The case $\gamma=1$ corresponds to a {linear} error bound
\[
    \mathfrak d(z)
    \;\le\;
    c\,\delta(z)
    \qquad \text{for all $z$ in a neighbourhood of $S$,}
\]
i.e., to bounded linear regularity / metric subregularity of the feasibility mapping.
For intersections of convex sets, this property is guaranteed under standard Slater-type
interiority assumptions. For example, if $S$ has nonempty interior then the H\"older regularity condition holds with exponent $\gamma=1$ and some constant $c>0$.

In this regime, Corollary~\ref{cor:rate-norm-gamma1} shows that the best-approximation
error $\|x_k - P_S(x)\|$ decays like $\alpha_k^{1/2}$. In particular, for the classical
Halpern stepsize $\alpha_k = 1/(k+1)$, the result gives the quantitative rate
\[
    \|x_k - P_S(x)\| = \mathcal O \bigl(k^{-1/2}\bigr),
\]
while the feasibility gap satisfies $\mathfrak d(x_k) = \mathcal O(\alpha_k)$ under
Theorem~\ref{thm:rate-distance}. Thus in the ``well-conditioned'' Slater case ($\gamma=1$) we recover an $\mathcal{O}(1/k)$ decay in the distance to the feasible set and an $\mathcal{O}(1/\sqrt{k})$ decay for the best approximation error, with fully explicit constants in terms of $c_0$, $c$, $\mathfrak d(x)$ and the asymptotic behaviour of $(\alpha_k)$.

\begin{corollary}[Complexity for harmonic stepsizes]\label{cor:harmonic-complexity}
Under the assumptions of Theorem~\ref{thm:rate-norm}, suppose in addition that
\[
\alpha_k = \frac{1}{k}, \qquad k \ge 1.
\]
Then the Halpern iterates satisfy
\[
\|x_k - P_S(x)\|
=
\mathcal O\bigl(k^{-\gamma/(4-2\gamma)}\bigr).
\]
Consequently, for every \(\varepsilon>0\), it suffices to take
\[
N
=
\mathcal O\bigl(\varepsilon^{-(4\gamma^{-1}-2)}\bigr)
\]
iterations to guarantee \(\|x_N - P_S(x)\|\le\varepsilon\).
\end{corollary}

\begin{proof}
Theorem~\ref{thm:rate-norm} shows that there exists a constant \(C>0\) such
that, for all \(k\) large enough,
\[
\|x_k - P_S(x)\|
\le
C\,\alpha_k^{\,\gamma/(4-2\gamma)}
=
C\,k^{-\gamma/(4-2\gamma)}.
\]
This gives
\[
\|x_k - P_S(x)\|
=
\mathcal O\bigl(k^{-\gamma/(4-2\gamma)}\bigr).
\]
To guarantee \(\|x_N - P_S(x)\|\le\varepsilon\), it is enough to choose \(N\)
such that
\[
C\,N^{-\gamma/(4-2\gamma)} \le \varepsilon,
\]
or equivalently
\[
N \ge \left(\frac{C}{\varepsilon}\right)^{4\gamma^{-1}-2}.
\]
Thus it suffices to take
\[
N
=
\mathcal O\bigl(\varepsilon^{-(4\gamma^{-1}-2)}\bigr)
\]
iterations.
\end{proof}

Iteration-complexity statements expressed through $\alpha_k$ are the natural form for Halpern's method,
since the stepsizes play the same role as smoothness constants in classical first-order optimization.  
Specializing to $\alpha_k = 1/k$ recovers the concrete $\mathcal O(\varepsilon^{-(4\gamma^{-1}-2)})$ rate,
but other schedules (e.g., $\alpha_k = 1/k^a$ with $a\in(0,1]$) can be handled identically.

\begin{remark}[Best choice in the harmonic family]\label{rem:best-harmonic}
Consider harmonic stepsizes
$\alpha_k$ given by
\[
    \alpha_k = \frac{1}{\mu k}, \qquad \mu\ge1.
\]
Then \(\alpha_k\in(0,1]\) for all \(k\ge1\), and
\[
    \frac{1}{\alpha_k} = \mu k
    \quad\Longrightarrow\quad
    \limsup_{k\to\infty} \ \Bigl(\tfrac{1}{\alpha_{k+1}}-\tfrac{1}{\alpha_k}\Bigr)=\mu.
\]
Recall that
\[
    p := \frac{1}{2\gamma^{-1}-1}
    \quad\Longrightarrow\quad
    \frac{\gamma}{4-2\gamma} = \frac{p}{2}.
\]
Hence \(\alpha_k^{\,\gamma/(4-2\gamma)}=\alpha_k^{p/2}=(1/(\mu k))^{p/2}=\mu^{-p/2}k^{-p/2}\), and
Theorem~\ref{thm:rate-norm} implies that, up to a constant independent of \(\mu\),
\[
    \limsup_{k\to\infty} \ k^{p/2}\,\|x_k-P_S(x)\|
    \;\lesssim\;
    \Phi(\mu),
    \qquad
    \Phi(\mu):=\frac{1}{\mu^{p/2}\sqrt{2-p\mu}},
    \quad 1\le\mu < 2.
\]
To minimize \(\Phi\), differentiate \(\log\Phi(\mu)\):
\[
    \frac{d}{d\mu}\log\Phi(\mu_* )
    = -\frac{p}{2\mu_* } + \frac{p/2}{2-p\mu_* } = 0 
    \quad \implies \quad 
    \mu_* = \frac{2}{p+1} = 2-\gamma.
\]
Since \(\gamma\in(0,1]\), we have \(\mu_*\in[1,2)\), so this choice is admissible.
Thus, within the harmonic family \(\alpha_k = 1/(\mu k)\), the choice
\[
    \alpha_k = \frac{1}{(2-\gamma)k}
\]
minimizes the asymptotic constant in the rate bound of Theorem~\ref{thm:rate-norm}.
\end{remark}

Although the convergence and rate results above are stated for a fixed operator~$T$, the arguments apply without modification to Halpern iterations driven by a sequence $(T_k)$, namely
\[
x_{k+1} = \alpha_k x + (1-\alpha_k) T_k x_k.
\]
If all $T_k$ share the same feasible set $S$ and satisfy the local decrease property~\eqref{eq:assT} with a uniform constant $c_0>0$, then Lemma~\ref{lem:contraction}, Lemma~\ref{lem:orderrec}, and the entire rate analysis remain valid.  
Thus the results in this section hold verbatim for operator sequences as long as the decrease inequality is uniform across~$k$.

\medskip
These findings reinforce that Halpern-type anchored schemes possess a subtle but powerful mechanism for convergence: the static anchor $x_0$ enforces directional consistency, while the stepsize and geometry jointly determine the rate. Unlike acceleration via inertial momentum, as seen in heavy-ball methods or inertial quasi-nonexpansive schemes~\cite{Shehu2022Inertial}, our results suggest that anchoring combined with H\"older-type error structure yields comparable complexity with better theoretical stability. In particular, for projection-type operators lacking global nonexpansiveness, Halpern's iteration not only ensures strong convergence but also admits rate bounds tightly governed by local regularity. This positions Halpern-type iterations as a geometrically robust alternative to momentum-based acceleration, especially in problems where feasible-set conditioning is poor or unknown.

\section{Rates of convergence for special instances of Halpern-type iterations}\label{sec:examples}

We consider the Halpern-type iterative scheme
\begin{equation}\label{halpernsecex}
x_{k+1} = \alpha_k x + (1-\alpha_k) T x_k,
\end{equation}
\[
\alpha_k \in (0,1], \qquad \alpha_k \to 0, \qquad  \limsup_{k\to\infty}\,\left(\frac{1}{\alpha_{k+1}}-\frac{1}{\alpha_k}\right)<2, \qquad \lim_{k\to\infty} \ \frac{\alpha_{k}}{\alpha_{k+1}} = 1,
\]
where \(x\) is a fixed anchor point and \(T\) denotes one of the projection-based operators described below.  
Depending on the choice of \(T\), we obtain several variants of the Halpern projection method.
We consider in this section
6 different choices for $T$
and we show that these operators satisfy Assumption \ref{ass:T}. Therefore for the Halpern iterations 
\eqref{halpernsecex} 
combined with these operators, the convergence Theorem 
\ref{thm:globalconv}
and the rates of convergence
given by Theorems
\ref{thm:rate-distance} and
\ref{thm:rate-norm} 
apply.

\subsection{Method of Alternating Projections}\label{secmap}

For closed convex sets \(U_1,\ldots,U_m \subset \RR^n\), the method of alternating projections (MAP) is defined by
\[
T_{\mathrm{MAP}}x := P_{U_m} P_{U_{m-1}} \cdots P_{U_1}x.
\]
The corresponding Halpern iteration is
\[
x_{k+1}
=
\alpha_k x
+
(1-\alpha_k)\,T_{\mathrm{MAP}}x_k.
\]

\paragraph{Verification of Assumption~\ref{ass:T}.}
Let \(p_0 := x\) and \(p_i := P_{U_i}p_{i-1}\) for \(i=1,\ldots,m\), so that \(p_m = T_{\mathrm{MAP}}x\) and \(p_i \in U_i\).  
Given any \(s\in S = \bigcap_{i=1}^m U_i\), the projection identity from Proposition \ref{prop:pythag} gives
\[
\|p_{i-1}-s\|^2 \;\ge\; \|p_i - s\|^2 + \|p_{i-1}-p_i\|^2
\qquad (i=1,\ldots,m)
\]
and yields, after summing over all \(i\),
\[
\|x-s\|^2 - \|T_{\mathrm{MAP}}x - s\|^2
\;\ge\;
\sum_{i=1}^m \|p_{i-1}-p_i\|^2.
\]
For any \(k\in\{1,\ldots,m\}\), the Cauchy--Schwarz inequality applied to the
numbers \(\|p_{i-1}-p_i\|\) gives
\begin{equation}\label{firstmap}
\sum_{i=1}^k \|p_{i-1}-p_i\|^2
\;\ge\;
\frac{1}{k}\Big(\sum_{i=1}^k \|p_{i-1}-p_i\|\Big)^2.
\end{equation}
By the triangle inequality,
for $k \leq m$, we have
\begin{equation}\label{secondmap}
\sum_{i=1}^k \|p_{i-1}-p_i\|
\;\ge\;
\Big\|\sum_{i=1}^k (p_{i-1}-p_i)\Big\|
=
\|p_0 - p_k\|
=
\|x-p_k\|.
\end{equation}
Combining \eqref{firstmap} and \eqref{secondmap} yields
\[
\sum_{i=1}^k \|p_{i-1}-p_i\|^2
\;\ge\;
\frac{1}{k}\|x-p_k\|^2
\;\ge\;
\frac{1}{m}\|x-p_k\|^2,
\]
since \(k\le m\).  
Because the left-hand side increases with \(k\), this implies
\[
\sum_{i=1}^m \|p_{i-1}-p_i\|^2
\;\ge\;
\frac{1}{m}\max_{1\le k\le m} \|x-p_k\|^2.
\]
Since \(p_k\in U_k\), we have \(\mathrm{dist}(x,U_k)\le \|x-p_k\|\) for each \(k\), hence
\[
\delta(x)^2
=
\max_{1\le i\le m}\mathrm{dist}(x,U_i)^2
\;\le\;
\max_{1\le k\le m}\|x-p_k\|^2.
\]
Combining the inequalities above gives the decrease estimate
\[
\|T_{\mathrm{MAP}}x - s\|^2
\;\le\;
\|x-s\|^2
\;-\;
\frac{1}{m}\,\delta(x)^2,
\qquad
\forall\,s\in S.
\]
Thus MAP satisfies Assumption~\ref{ass:T} with constant \(c_0 = 1/m\).  

\subsection{Cimmino's Method}\label{seccim}

Cimmino's method replaces sequential projections by their arithmetic average.  
For closed convex sets \(U_1,\ldots,U_m\subset\RR^n\), define
\[
T_{\mathrm{Cim}}x
:= 
\frac{1}{m}\sum_{k=1}^m P_{U_k}x,
\qquad
p_k := P_{U_k}x.
\]
The Halpern iteration becomes
\[
x_{k+1}
=
\alpha_k x + (1-\alpha_k)T_{\mathrm{Cim}}x_k.
\]
\paragraph{Verification of Assumption~\ref{ass:T}.}
Fix \(s\in S = \bigcap_{k=1}^m U_k\).  
Since \(s\in U_k\) and \(p_k = P_{U_k}x\), the projection identity from Proposition \ref{prop:pythag} gives
\[
\|x-s\|^2 \;\ge\; \|p_k - s\|^2 + \|x - p_k\|^2
\qquad(k=1,\ldots,m).
\]
Averaging over \(k\) yields
\[
\|x-s\|^2
\;\ge\;
\frac{1}{m}\sum_{k=1}^m \|p_k - s\|^2
\;+\;
\frac{1}{m}\sum_{k=1}^m \|x-p_k\|^2.
\]
By convexity of the squared norm,
\[
\frac{1}{m}\sum_{k=1}^m \|p_k - s\|^2
\;\ge\;
\biggl\|\frac{1}{m}\sum_{k=1}^m p_k - s\biggr\|^2
=
\|T_{\mathrm{Cim}}x - s\|^2.
\]
Thus
\[
\|x-s\|^2 - \|T_{\mathrm{Cim}}x - s\|^2
\;\ge\;
\frac{1}{m}\sum_{k=1}^m \|x-p_k\|^2.
\]
Since each \(p_k\in U_k\), we have
\[
\mathrm{dist}(x,U_k) \le \|x-p_k\|,
\qquad
\delta(x) = \max_{1\le k\le m}\mathrm{dist}(x,U_k),
\]
and therefore
\[
\frac{1}{m}\sum_{k=1}^m \|x-p_k\|^2
\;\ge\;
\frac{1}{m}\max_{1\le k\le m}\|x-p_k\|^2
\;\ge\;
\frac{1}{m}\,\delta(x)^2.
\]
We conclude that Cimmino's operator satisfies
\[
\|T_{\mathrm{Cim}}x - s\|^2
\;\le\;
\|x-s\|^2
\;-\;
\frac{1}{m}\,\delta(x)^2,
\qquad
\forall\,s\in S,\ \forall\,x\in\RR^n,
\]
that is, Assumption~\ref{ass:T} holds with constant \(c_0 = 1/m\).

\subsection{Parallel Polyhedral Projection Method (3PM)}
\label{3pm}

Given closed convex sets \(U_1,\ldots,U_m \subset \RR^n\), for each \(U_i\), define the projection \(P_{U_i}x.\) The supporting halfspace at \(P_{U_i}x\) is
\[
H_i(x)
:=
\bigl\{
y \in \RR^n : \langle x - P_{U_i}x,\, y - P_{U_i}x \rangle \le 0
\bigr\},
\]
which always satisfies \(U_i \subset H_i(x)\).  
The polyhedral outer approximation of the feasible region \(S=\cap_{i=1}^m U_i\) is
\[
\Omega(x) := \bigcap_{i=1}^m H_i(x).
\]
The 3PM operator is the projection of \(x\) onto this outer approximation:
\[
T_{\mathrm{3PM}}x
:=
P_{\Omega(x)}x
=
\argmin_{y\in\Omega(x)} \tfrac12\|y-x\|^2.
\]

\paragraph{Verification of Assumption~\ref{ass:T}.}
A key geometric feature of 3PM is that  
\[
S \subset H_i(x)\quad\forall\,i,
\qquad\text{hence } S \subset \Omega(x)\quad\forall\,x.
\]
Thus every feasible point \(s\in S\) belongs to \(\Omega(x)\), and by the projection inequality,
\[
\langle x - T_{\mathrm{3PM}}x,\, s - T_{\mathrm{3PM}}x \rangle \le 0.
\]
Expanding the square then gives the fundamental descent relation
\begin{equation}\label{eq:3pm-Pythag}
\|x-s\|^2
\;\ge\;
\|x - T_{\mathrm{3PM}}x\|^2 
\,+\,
\|T_{\mathrm{3PM}}x-s\|^2,
\qquad \forall\, s\in S.
\end{equation}

\medskip

\noindent
Next, we relate the step length \(\|x - T_{\mathrm{3PM}}(x)\|\) to the maximal violation \(\delta(x)\).  
Because \(U_i \subset H_i(x)\), the projection onto the intersection satisfies
\[
\|x - T_{\mathrm{3PM}}x\|
=
\mathrm{dist}(x,\Omega(x))
\;\ge\;
\mathrm{dist}(x,H_i(x))
\qquad\forall\, i.
\]
Since \(P_{U_i}x\in U_i\subset H_i(x)\),
\[
\mathrm{dist}(x,H_i(x))
=
\|x - P_{U_i}x\|
=
\mathrm{dist}(x,U_i).
\]
Taking the maximum over \(i\) yields
\[
\|x - T_{\mathrm{3PM}}x\|
\;\ge\;
\delta(x)
:=
\max_{1\le i\le m}\mathrm{dist}(x,U_i).
\]
Substituting this into \eqref{eq:3pm-Pythag} gives, for every \(s\in S\),
\[
\|x-s\|^2
\;\ge\;
\delta(x)^2 + \|T_{\mathrm{3PM}}x-s\|^2,
\]
or equivalently,
\[
\|T_{\mathrm{3PM}}x-s\|^2
\;\le\;
\|x-s\|^2
-
\delta(x)^2.
\]
Thus 3PM satisfies the local decrease condition of Assumption~\ref{ass:T} with the constant
\[
c_0 = 1.
\]

\subsection{Approximate Parallel Polyhedral Projection Method (A3PM)}\label{a3pm}

Given an accuracy parameter \(\varepsilon\in[0,1)\), each exact projection
\(P_{U_i}x\) is replaced by an approximate projection
\(\widehat P_{U_i}x := \widehat P_{U_i}(x,\varepsilon)\) satisfying
\[
\mathrm{dist}(\widehat P_{U_i}x,U_i)
\le
\varepsilon\,\mathrm{dist}(x,U_i),
\qquad
U_i \subseteq
\Bigl\{
y : \langle y-\widehat P_{U_i}x,\, x-\widehat P_{U_i}x\rangle \le 0
\Bigr\}.
\]
Each \(\widehat P_{U_i}x\) defines an approximate supporting halfspace
\[
\widehat H_i(x)
:=
\bigl\{
y : \langle x-\widehat P_{U_i}x,\, y-\widehat P_{U_i}x\rangle \le 0
\bigr\},
\]
and the approximate polyhedral outer approximation
\[
\widehat\Omega(x)
:=
\bigcap_{i=1}^m \widehat H_i(x).
\]
The A3PM operator is the (approximate) projection of \(x\) onto
\(\widehat\Omega(x)\):
\[
T_{\mathrm{A3PM}}x
:=
\widehat P_{\widehat\Omega(x)}(x,\varepsilon).
\]
The corresponding Halpern update becomes
\[
x_{k+1}
=
\alpha_k x
+
(1-\alpha_k)\,T_{\mathrm{A3PM}}x_k.
\]
\paragraph{Verification of Assumption~\ref{ass:T}.}
The geometry is analogous to 3PM: since \(U_i\subseteq \widehat H_i(x)\) for every \(i\),  
\[
S=\bigcap_{i=1}^m U_i \subseteq \widehat\Omega(x),
\qquad \forall\, x\in\RR^n.
\]
Thus every feasible point \(s\in S\) lies in \(\widehat\Omega(x)\), and the approximate projection inequality gives
\[
\langle x - T_{\mathrm{A3PM}}x,\, s - T_{\mathrm{A3PM}}x\rangle \le 0.
\]
Expanding the square yields the same Pythagorean decrease as in the exact 3PM case:
\begin{equation}\label{eq:a3pm-fejer}
\|x-s\|^2
\;\ge\;
\|x - T_{\mathrm{A3PM}}x\|^2
+
\|T_{\mathrm{A3PM}}x-s\|^2,
\qquad \forall\,s\in S.
\end{equation}
To relate the step size to the violation \(\delta(x)\), we use the approximate-projection bound  
\cite[Lemma 3.2]{Barros3pm}:
\[
\|x-\widehat P_X(x,\varepsilon)\|
\;\ge\;
(1-\varepsilon)\,\mathrm{dist}(x,X)
\qquad \forall \; x \in \mathbb{R}^n, \, X \subset \mathbb{R}^n.
\]
Applying this with \(X=\widehat\Omega(x)\) gives
\[
\|x - T_{\mathrm{A3PM}}x\|
\;\ge\;
(1-\varepsilon)\,\mathrm{dist}(x,\widehat\Omega(x)).
\]
Since \(U_i\subseteq \widehat H_i(x)\), 
$\widehat\Omega(x)\subseteq \widehat H_i(x)\), and  
\(\widehat P_{U_i}x\in\widehat H_i(x)\), we have
\[
\mathrm{dist}(x,\widehat\Omega(x))
\;\ge\;
\mathrm{dist}(x,\widehat H_i(x))
=
\|x-\widehat P_{U_i}x\|
\qquad\forall \ i.
\]
A second application of \cite[Lemma 3.2]{Barros3pm} with \(X=U_i\) yields
\[
\|x - \widehat P_{U_i}x\|
\;\ge\;
(1-\varepsilon)\,\mathrm{dist}(x,U_i),
\qquad\forall \ i,
\]
and therefore
\[
\mathrm{dist}(x,\widehat\Omega(x))
\;\ge\;
(1-\varepsilon)\,\delta(x).
\]
Combining with the bound for \(\|x - T_{\mathrm{A3PM}}x\|\) gives
\[
\|x - T_{\mathrm{A3PM}}x\|
\;\ge\;
(1-\varepsilon)^2\,\delta(x).
\]
Substituting this into \eqref{eq:a3pm-fejer} yields, for every \(s\in S\),
\[
\|T_{\mathrm{A3PM}}x-s\|^2
\;\le\;
\|x-s\|^2
-
(1-\varepsilon)^4\,\delta(x)^2.
\]
Thus A3PM satisfies the decrease property of Assumption~\ref{ass:T} with
\[
c_0 = (1-\varepsilon)^4.
\]

\subsection{Successive Centralized CRM (SCCRM)}\label{sccrm}

We recall the SCCRM operator introduced in \cite{succccrm}.  
For convex sets \(A,B \subset \mathbb{R}^n\), define
\[
Z_{A,B} = P_A \circ P_B, 
\qquad
\tilde{Z}_{A,B} = \tfrac{1}{2}(P_A + P_B), 
\qquad
\bar{Z}_{A,B} = \tilde{Z}_{A,B} \circ Z_{A,B}.
\]
Let \(R_A = 2P_A - I\) and \(R_B = 2P_B - I\), and denote by
\[
\mathcal{C}_{A,B}(x)
=
\operatorname{circ}\left(x, R_A(x), R_B(x)\right)
\]
the circumcenter of the three points \(x\), \(R_A(x)\), and \(R_B(x)\).  
The basic two-set SCCRM operator is then
\[
T_{A,B}(x)
=
\mathcal{C}_{A,B} \bigl(\bar{Z}_{A,B}(x)\bigr).
\]
At iteration \(k\), SCCRM selects a pair of sets
\((U_{\ell(k)},U_{r(k)})\) according to control sequences \(\ell(k)\) and \(r(k)\), for instance 
\[ \ell(k)=1,2,\ldots,m,1,2,\ldots, \qquad r(k)=2,3,\ldots,m,1,2,\ldots, \]
and performs
\[
x_{k+1}
=
T_{U_{r(k)},U_{\ell(k)}}x_k.
\]
Within the Halpern scheme \eqref{eq:halpern}, a full cycle of
\(m-1\) such two-set updates defines the operator
\[
T_{\mathrm{SCCRM}}
:=
T_{U_{r(m-1)},U_{\ell(m-1)}} \circ \cdots
\circ T_{U_{r(1)},U_{\ell(1)}},
\qquad
x_{k+1}
=
\alpha_k x + (1-\alpha_k)\,T_{\mathrm{SCCRM}}x_k.
\]

\paragraph{Verification of Assumption~\ref{ass:T}.}
Let \(s\in S=\bigcap_i U_i\). Define
\[
y_0:=x,
\qquad 
y_j:=T_{U_{j+1},U_j}y_{j-1},
\qquad j=1,\ldots,m-1,
\]
so that \(y_{m-1}=T_{\mathrm{SCCRM}}x\). For each two-set SCCRM step, let \(a_j\) and \(b_j\) denote the two successive
projection points, and let \(c_j\) denote the centralized point from which the
circumcenter is computed:
\[
a_j:=P_{U_j}y_{j-1},
\qquad
b_j:=P_{U_{j+1}}a_j,
\qquad
c_j:=\bar Z_{U_{j+1},U_j}(y_{j-1}).
\]
Set
\[
\Delta_j:=\|y_{j-1}-s\|^2-\|y_j-s\|^2 .
\]
Since \(s\in U_j\cap U_{j+1}\), the projection identity from
Proposition~\ref{prop:pythag}, applied first to \(U_j\) and then to
\(U_{j+1}\), gives
\[
\|y_{j-1}-s\|^2
\ge
\|b_j-s\|^2
+
\|y_{j-1}-a_j\|^2
+
\|a_j-b_j\|^2 .
\]
Moreover, by Lemma~2.4 of \cite{BehlingBelloCruzIusemSantos2024},
\[
\|b_j-s\|^2-\|c_j-s\|^2
\ge
\frac12\|b_j-P_{U_j}b_j\|^2
=
2\|b_j-c_j\|^2,
\]
and, by Lemma~2.5 of \cite{BehlingBelloCruzIusemSantos2024},
\[
\|y_j-s\|^2
\le
\|c_j-s\|^2-\|c_j-y_j\|^2.
\]
Combining these inequalities yields
\begin{equation}\label{eq:sccrm-residuals}
\Delta_j
\ge
\|y_{j-1}-a_j\|^2
+
\|a_j-b_j\|^2
+
\|b_j-c_j\|^2
+
\|c_j-y_j\|^2.
\end{equation}
We now relate the decrease to \(\delta(x)\). Consider the polygonal path
\begin{align*}
    x=y_0\to a_1\to b_1\to c_1\to y_1
\to a_2\to b_2\to c_2\\\to y_2
\to\cdots\to a_{m-1}\to b_{m-1}\to c_{m-1}\to y_{m-1}.
\end{align*}
This path visits all sets: \(a_1\in U_1\), and \(b_{i-1}\in U_i\) for
\(i=2,\ldots,m\). Hence, by the triangle inequality,
\[
\delta(x)
=
\max_{1\le i\le m}\dist(x,U_i)
\le
\sum_{j=1}^{m-1}
\left(
\|y_{j-1}-a_j\|
+
\|a_j-b_j\|
+
\|b_j-c_j\|
+
\|c_j-y_j\|
\right).
\]
By Cauchy--Schwarz and \eqref{eq:sccrm-residuals},
\begin{align*}
    \delta(x)^2
&\le
4(m-1)
\sum_{j=1}^{m-1}
\left(
\|y_{j-1}-a_j\|^2
+
\|a_j-b_j\|^2
+
\|b_j-c_j\|^2
+
\|c_j-y_j\|^2
\right)
\\&\le
4(m-1)\sum_{j=1}^{m-1}\Delta_j .
\end{align*}
Finally, telescoping gives
\[
\sum_{j=1}^{m-1}\Delta_j
=
\|x-s\|^2-\|T_{\mathrm{SCCRM}}x-s\|^2.
\]
Therefore,
\[
\|T_{\mathrm{SCCRM}}x-s\|^2
\le
\|x-s\|^2
-
\frac{1}{4(m-1)}\,\delta(x)^2.
\]
Thus SCCRM satisfies Assumption~\ref{ass:T} with \(c_0=1/(4(m-1))\).

\subsection{CRM in product space (CRM)}\label{crm}

We use the standard product-space reformulation of the feasibility problem
\(S=\bigcap_{i=1}^m U_i\).
Fix the Halpern anchor \(x\in\RR^n\) and define
\[
z_0:=(x,x,\ldots,x)\in\RR^{nm}.
\]
Let
\[
W:=U_1\times U_2\times\cdots\times U_m
\subseteq \RR^{nm},
\qquad
D:=\{(u,u,\ldots,u)\in\RR^{nm}: u\in\RR^n\}.
\]
Then \(W\cap D=\{(s,\ldots,s): s\in S\}\).

For \(z=(x^{(1)},x^{(2)},\ldots,x^{(m)})\in\RR^{nm}\), the projection onto \(D\) is
\[
P_D(z)=(\bar{x},\bar{x},\ldots,\bar{x}),
\qquad
\bar{x}:=\frac{1}{m}\sum_{i=1}^m x^{(i)},
\]
and the projection onto \(W\) is
\[
P_W(z)=\bigl(P_{U_1}(x^{(1)}),\,P_{U_2}(x^{(2)}),\,\ldots,\,P_{U_m}(x^{(m)})\bigr).
\]
Define the reflectors \(R_W:=2P_W-I\) and \(R_D:=2P_D-I\). The CRM operator in the
product space is
\[
T_{\mathrm{CRM}}(z)
:=
\operatorname{circ}\bigl(z,\;R_W(z),\;R_D(R_W(z))\bigr).
\]
The corresponding Halpern update in the product space is
\[
z_{k+1}
=
\alpha_k z_0 + (1-\alpha_k)\,T_{\mathrm{CRM}}z_k.
\]

We note that \(z_k\in D\) for all \(k\). Indeed, \(z_0\in D\) by construction and
\cite[Lemma~3]{crmprod} (with \(K=W\) and \(U=D\)) ensures that \(T_{\mathrm{CRM}}(z)\in D\)
whenever \(z\in D\). Since \(D\) is affine (hence convex), it follows that
\(\alpha_k z_0+(1-\alpha_k)T_{\mathrm{CRM}}z_k\in D\) whenever \(z_k\in D\), proving the claim.

When an iterate in \(\RR^n\) is needed, we recover it from the diagonal:
\[
x_k := (P_D z_k)^{(1)}=\frac{1}{m}\sum_{i=1}^m x_k^{(i)},
\quad\text{where } z_k=\left(x_k^{(1)},\ldots,x_k^{(m)}\right).
\]

\paragraph{Verification of Assumption~\ref{ass:T}.}
Let \(z\in D\). Define
\[
C(z):=\operatorname{circ} \bigl(z,\;R_W(z),\;R_D(R_W(z))\bigr)
= T_{\mathrm{CRM}}(z).
\]
We apply \cite[Lemma~3]{crmprod} with \(K=W\) and \(U=D\). Since \(W\cap D\neq\emptyset\),
\(C(z)\) is well-defined for every \(z\in D\), satisfies \(C(z)\in D\), and
\[
C(z)=P_{H_z\cap D}(z),
\quad
H_z:=
\begin{cases}
\{y\in\RR^{nm}:\ \langle y-P_W(z),\, z-P_W(z)\rangle=0\}, & z\notin W,\\[2mm]
W, & z\in W.
\end{cases}
\]
Moreover, by \cite[Lemma~5]{crmprod}, for any \(s\in W\cap D\),
\begin{equation}\label{eq:crm-fqne}
\|C(z)-s\|^2 \le \|z-s\|^2 - \|z-C(z)\|^2 .
\end{equation}

Set \(S:=W\cap D\) and define the violation (for the two-set problem \(W\cap D\))
\[
\delta(z):=\max\{\dist(z,W),\dist(z,D)\}.
\]
Since \(z\in D\), we have \(\delta(z)=\dist(z,W)\).

If \(z\in W\) then \(\delta(z)=0\) and \(C(z)=z\), so Assumption~\ref{ass:T} is trivially satisfied with $c_0=1$.

Assume \(z\notin W\).
Since \(C(z)=P_{H_z\cap D}(z)\in H_z\cap D\subseteq H_z\), we obtain
\[
\|z-C(z)\|\ \ge\ \dist(z,H_z).
\]
Using the point-to-hyperplane distance formula and the definition of \(H_z\),
\[
\dist(z,H_z)
=
\frac{|\langle z-P_W(z),\, z-P_W(z)\rangle|}{\|z-P_W(z)\|}
=
\|z-P_W(z)\|
=
\dist(z,W)
=
\delta(z).
\]
Hence \(\|z-C(z)\|\ge \delta(z)\). Substituting into \eqref{eq:crm-fqne} yields, for all \(s\in S\),
\[
\|T_{\mathrm{CRM}}(z)-s\|^2
=
\|C(z)-s\|^2
\le
\|z-s\|^2-\delta(z)^2,
\]
so Assumption~\ref{ass:T} holds with \(c_0=1\).

\section{Numerical experiments: accelerating Dykstra's algorithm}
\label{sec:numexp}

In this section, we numerically compare Dykstra's algorithm with
Halpern-type iterations of the form \eqref{halpernsecex}, where \(T\)
is taken to be one of the operators
\[
T_{\mathrm{MAP}},\quad
T_{\mathrm{Cim}},\quad
T_{\mathrm{3PM}},\quad
T_{\mathrm{A3PM}},\quad
T_{\mathrm{SCCRM}},\quad
T_{\mathrm{CRM}}.
\]
These operators are defined in Sections~\ref{secmap}, \ref{seccim},
\ref{3pm}, \ref{a3pm}, \ref{sccrm}, and \ref{crm}, respectively.

For the sequences \((x_k)\) generated with \(T=T_{\mathrm{Cim}}\) and
\(T=T_{\mathrm{A3PM}}\), we also consider parallel implementations of the
projections required to compute \(T_{\mathrm{Cim}}x_k\) and
\(T_{\mathrm{A3PM}}x_k\), respectively. In the tables below, these variants
are denoted by Cimmino~\(\parallel\) and A3PM~\(\parallel\).

The goal of all methods is to find
a point in $S$ that is the closest
to the initial anchor point $x$.
We consider two setups for $S$:
an intersection of ellipsoids and an intersection of polyhedra.
We implemented all methods in Julia
and the corresponding code is available
at {\url{https://github.com/vguigues/Halpern}}.

For the first experiment (with ellipsoids) we have $m$ ellipsoids $U_1,$ $\ldots,$ $U_m$ in $\mathbb{R}^n$, where ellipsoid $U_i$ is centered
at $y_i \in \mathbb{R}^n$ and is given by
\begin{equation}
U_i=\{y \in \mathbb{R}^n: (y-y_i)^T Q_i (y-y_i) \leq \eta_i^2\},
\end{equation}
for some positive $\eta_i$ and positive definite matrices $Q_i$.
We generate the centers
$y_i$ randomly 
and $Q_i$ of the form
$Q_i=A_i A_i^T + \lambda_i I_n$, where
matrices $A_i$ and positive $\lambda_i$ are generated randomly too. 
In that manner, matrices $Q_i$ are positive definite, as desired. We also choose $\eta_i$ sufficiently large so that the intersection of sets $U_i$ contains
the ball
$B=\{x: \|x\|_2 \leq \theta\}$ for some parameter $\theta>0$, namely
\begin{equation}\label{eqki}
\eta_i \geq (\theta+\|y_i\|_2) \sqrt{\|Q_i\|_2}
\end{equation}
for all $i$ (in particular, the intersection $S=\bigcap_{i=1}^m U_i$ is nonempty). Indeed, if \eqref{eqki} is satisfied and $x$ belongs to $B$, then
$$
(x-y_i)^T Q_i (x-y_i) \leq \|x-y_i\|_2^2 \|Q_i\|_2 \leq (\theta+\|y_i\|_2)^2 \|Q_i\|_2 \leq 
\eta_i^2
$$
and therefore $x \in U_i$. Thus \(B\subset S\), and in particular \(S\) has nonempty interior. Hence,
by the classical bounded linear regularity result for finite collections of
closed convex sets; see, e.g., \cite[Corollary~5.14]{BauschkeBorwein1996},
the family \(\{U_i\}_{i=1}^m\) is linearly regular on bounded sets.
Consequently, \eqref{eq:holder-eb} holds for these ellipsoidal instances
with exponent \(\gamma=1\).

For the polyhedral experiments, the sets are generated as
\[
U_i=\{x:Ax\leq Ax^\star+\alpha \xi_i\},
\]
where \(A\) is a given matrix, \(x^\star\) is a fixed vector, \(\alpha>0\),
and \(\xi_i\) has independent random entries in \([0,1]\). This construction
ensures that the feasible intersection is nonempty, since \(x^\star\in U_i\)
for every \(i\), and hence \(x^\star\in S:=\bigcap_{i=1}^m U_i\). For these polyhedral instances, \eqref{eq:holder-eb} follows from
Hoffman's error bound applied to the finite system of linear inequalities
defining \(S\). In particular, finite collections of polyhedra with
nonempty intersection are linearly regular on bounded sets; hence
\eqref{eq:holder-eb} holds with exponent \(\gamma=1\), corresponding to
the linear regularity case in Definition~\ref{def:holder-regularity}.
The initial points \(x_0\) are generated outside \(S\).

For completeness, we recall Dykstra's algorithm
\cite{dykstra1983algorithm,boyle1986method}, which we implemented in Julia
and use below as a benchmark. The algorithm generates sequences
\((x_i^k)\) and \((y_i^k)\) through the updates
\[
\begin{cases}
x_0^{\,k} = x_m^{\,k-1}, \\[6pt]
x_i^{\,k} = P_{U_i}\left(x_{i-1}^{\,k} - y_i^{\,k-1}\right),
    & i=1,2,\dots,m, \\[6pt]
y_i^{\,k} = x_i^{\,k} - \left(x_{i-1}^{\,k} - y_i^{\,k-1}\right),
    & i=1,2,\dots,m,
\end{cases}
\qquad k=1,2,\dots,
\]
with initializations \(x_m^{\,0}=x^0\) and \(y_i^{\,0}=0\).
It is well known \cite{hundal1997two, boyle1986method, bauschke1994dykstra} that each subsequence \(x_i^k\) converges to the best
approximation point \(P_S x^0\), where \(S=\bigcap_{i=1}^m U_i\).

The speed of convergence of Dykstra's algorithm is quite delicate.  
In full generality (for arbitrary closed convex sets), one can only guarantee sublinear decay of the error, and in fact the algorithm may converge arbitrarily slowly.  
By contrast, for polyhedral intersections, Dykstra's method enjoys linear convergence \cite{DeutschHundal1994, deutsch1995dykstra}.  
Besides polyhedral sets, Dykstra's projection method enjoys linear convergence for a
larger class of sets, including ellipsoids, under an additional strict complementarity condition; see \cite{WangPong2024}.
Recent work has also highlighted the possibility of very long ``stalling'' phases: Bauschke et al.~\cite{bauschke2020dykstra} construct a simple two-dimensional example (a line and a square) for which, by a suitable choice of starting point, the iterates remain almost stationary for an arbitrarily large number of cycles before finally converging.  

These results show that, while Dykstra's algorithm always converges to the best approximation point under mild assumptions, its practical efficiency can range from robust linear convergence in well-conditioned (e.g., polyhedral) settings to extremely slow progress in poorly conditioned configurations. This sensitivity of the convergence speed to geometric regularity motivates the development of alternative projection-type schemes with explicit rate guarantees under error-bound assumptions, such as the anchored Halpern framework studied in this paper. However, Dykstra's method remains one of the most widely used projection-based
solvers for convex feasibility and best-approximation problems. For this
reason, we include Dykstra's algorithm as a reference method in the numerical
experiments. This comparison situates the
performance of Halpern-type schemes relative to a classical algorithm that is
known to converge globally but typically exhibits slower asymptotic decay in the
absence of strong regularity.

For the ellipsoidal experiments, we consider eight instances in total. Four
instances are generated with \(\theta=1\) and
\[
(m,n)=(10,10),\ (20,10),\ (20,20),\ (20,100),
\]
and four instances are generated with \(\theta=0.01\) and the same choices of
\((m,n)\). We test the two stepsize sequences \(\alpha_k=1/k\) and
\(\alpha_k=1/\sqrt{k}\), both of which satisfy Assumption~\ref{ass:alphan}. For each
instance, the reference solution \(s^\star=P_Sx_0,\) the projection of the initial point onto \(S\), is computed using Gurobi's
quadratic solver. All methods are stopped at the first iterate \(x_k\)
satisfying
\[
\|x_k-s^\star\|\leq \varepsilon,
\qquad \varepsilon=0.01.
\]
A method is also stopped if it does not find such an iterate within
10 minutes. The number of iterations and CPU time for all methods and
instances are reported in Table~\ref{tableell1} for \(\theta=0.01\) and in
Table~\ref{tableell2} for \(\theta=1\).

We also compare the methods on the problem of finding a point in the
intersection of \(m\) polyhedra in \(\mathbb{R}^n\), each described by \(k\)
linear inequalities. The corresponding results are reported in
Tables~\ref{tableresp} and~\ref{tablesresp2}. In these experiments, for each
problem size, each method is run 10 times, and the reported CPU time is the
average CPU time over these runs.

Overall, Halpern MAP and Halpern Cimmino are the slowest methods. By
contrast, A3PM and A3PM~\(\parallel\) with \(\alpha_k=1/k\) are generally
the fastest methods in terms of CPU time. They also exhibit robust behavior
across the reported instances: they reach the prescribed tolerance in all
tests, whereas Dykstra's algorithm hits the time limit in some of the larger
ellipsoidal instances. Although Dykstra's algorithm is competitive in several
cases, A3PM and A3PM~\(\parallel\) give the best CPU times in most of the
experiments.

Previous Halpern-type accelerations based on projection-type operators had only used the classical choice \(T=T_{\mathrm{MAP}}\). The present
experiments show that other such operators, notably
\(T=T_{\mathrm{A3PM}}\), can also be effectively embedded into Halpern
iterations. Together with the theory developed in this paper, this provides
both convergence with explicit rates under H\"older error bounds and strong
computational performance. In particular, Halpern A3PM with
\(\alpha_k=1/k\) appears to be a fast and robust method for solving the best
approximation problem \eqref{BAP}.

\section{Conclusion}

The above results extend classical convergence analyses of projection methods and Halpern-type iterations to settings governed by H\"older-type regularity.
The framework captures both linear and sublinear regimes within a unified formulation.
It also accommodates operator sequences that vary across iterations, removing the need for continuity assumptions traditionally imposed in the literature.

Our numerical experiments show that Halpern-type iterations combined with some of the projection-type operators
from Section \ref{sec:examples}
are quicker than Dykstra's
algorithm to find the projection of a point
in an intersection of ellipsoids or in an intersection of polyhedra.

\begin{table} \centering
\caption{Comparison of iterations and CPU time (s) for Dykstra's algorithm and Halpern iterations driven by various operators: A3PM ($T=T_{\mathrm{A3PM}}$), MAP ($T_{\mathrm{MAP}}$), Cimmino ($T_{\mathrm{Cim}}$), SCCRM ($T_{\mathrm{SCCRM}}$), CRM ($T_{\mathrm{CRM}}$), and 3PM ($T_{\mathrm{3PM}}$). Parallel implementations are denoted by $\parallel$. All methods are run with $\varepsilon=10^{-2}$ and $\theta=0.01$. Intersection of ellipsoids.}\label{tableell1}
\vspace{5pt}
\scalebox{0.77}{
\begin{tabular}{|c|c|c|c|c|c|c|}
\hline
Method & $m$ & $n$ & Iter, $\alpha_k=1/k$ & Time, $\alpha_k=1/k$ & Iter, $\alpha_k=1/\sqrt{k}$ & Time, $\alpha_k=1/\sqrt{k}$ \\
\hline
\hline
Halpern A3PM & 10 & 10 & 440 & \textbf{0.02} & 177544 & 7.11 \\
Halpern A3PM $\parallel$ & 10 & 10 & 440 & \textbf{0.02} & 177544 & 6.69 \\
Halpern 3PM & 10 & 10 & 427 & 1.13 & 35110 & 87.65 \\
Halpern 3PM $\parallel$ & 10 & 10 & 427 & 1.35 & 35110 & 107.18 \\
Halpern MAP & 10 & 10 & 18222663 & 600 & 937026 & 600 \\
Halpern Cimmino & 10 & 10 & 16260302 & 600 & 293753 & 600 \\
Halpern Cimmino $\parallel$ & 10 & 10 & 10240611 & 600 & 283736 & 600 \\
Halpern SCCRM & 10 & 10 & 10944267 & 600 & 1605706 & 600 \\
Halpern CRM & 10 & 10 & 2512449 & 600 & 173264 & 600 \\
Dykstra & 10 & 10 & \textbf{28} & 0.06 & \textbf{28} & \textbf{0.05} \\
\hline \hline
Halpern A3PM & 20 & 10 & 482 & \textbf{0.03} & 208198 & 13.30 \\
Halpern A3PM $\parallel$ & 20 & 10 & 482 & 0.04 & 208198 & 11.13 \\
Halpern 3PM & 20 & 10 & 392 & 1.26 & 176290 & 424.96 \\
Halpern 3PM $\parallel$ & 20 & 10 & 392 & 0.91 & 176290 & 531.59 \\
Halpern MAP & 20 & 10 & 8062077 & 600 & 653562 & 600 \\
Halpern Cimmino & 20 & 10 & 6560883 & 600 & 218716 & 600 \\
Halpern Cimmino $\parallel$ & 20 & 10 & 5464958 & 600 & 270053 & 600 \\
Halpern SCCRM & 20 & 10 & 6250199 & 600 & 2430066 & 600 \\
Halpern CRM & 20 & 10 & 769149 & 600 & 137824 & 600 \\
Dykstra & 20 & 10 & \textbf{16} & 0.05 & \textbf{16} & \textbf{0.03} \\
\hline \hline
Halpern A3PM & 20 & 20 & 589 & 0.11 & 293931 & 27.36 \\
Halpern A3PM $\parallel$ & 20 & 20 & 589 & \textbf{0.08} & 293931 & 24.13 \\
Halpern 3PM & 20 & 20 & 306 & 2.14 & 41506 & 179.57 \\
Halpern 3PM $\parallel$ & 20 & 20 & 306 & 1.34 & 41506 & 234.83 \\
Halpern MAP & 20 & 20 & 7060660 & 600 & 409114 & 600 \\
Halpern Cimmino & 20 & 20 & 6509738 & 600 & 117741 & 600 \\
Halpern Cimmino $\parallel$ & 20 & 20 & 3404319 & 600 & 101044 & 600 \\
Halpern SCCRM & 20 & 20 & 7176725 & 600 & 1980609 & 600 \\
Halpern CRM & 20 & 20 & 656318 & 600 & 85185 & 600 \\
Dykstra & 20 & 20 & \textbf{22} & 0.11 & \textbf{22} & \textbf{0.12} \\
\hline \hline
Halpern A3PM & 20 & 100 & 401 & 0.47 & 138535 & 123.05 \\
Halpern A3PM $\parallel$ & 20 & 100 & 401 & \textbf{0.44} & 138535 & 109.98 \\
Halpern 3PM & 20 & 100 & \textbf{47} & 0.92 & \textbf{542} & \textbf{9.18} \\
Halpern 3PM $\parallel$ & 20 & 100 & \textbf{47} & 1.73 & \textbf{542} & 16.58 \\
Halpern MAP & 20 & 100 & 685319 & 600 & 137913 & 600 \\
Halpern Cimmino & 20 & 100 & 379666 & 600 & 38446 & 600 \\
Halpern Cimmino $\parallel$ & 20 & 100 & 156341 & 600 & 21293 & 600 \\
Halpern SCCRM & 20 & 100 & 998620 & 600 & 388228 & 600 \\
Halpern CRM & 20 & 100 & 48080 & 600 & 21883 & 600 \\
Dykstra & 20 & 100 & 32553 & 600 & 40206 & 600 \\
\hline
\end{tabular}
}
\end{table}

\begin{table} \centering
\caption{Comparison of iterations and CPU time (s) for Dykstra's algorithm and Halpern iterations driven by various operators: A3PM ($T=T_{\mathrm{A3PM}}$), MAP ($T_{\mathrm{MAP}}$), Cimmino ($T_{\mathrm{Cim}}$), SCCRM ($T_{\mathrm{SCCRM}}$), CRM ($T_{\mathrm{CRM}}$), and 3PM ($T_{\mathrm{3PM}}$). Parallel implementations are denoted by $\parallel$. All methods are run with $\varepsilon=10^{-2}$ and $\theta=1$. Intersection of ellipsoids.}
\vspace{5pt}
\scalebox{0.77}{
\begin{tabular}{|c|c|c|c|c|c|c|}
\hline
Method & $m$ & $n$ & Iter, $\alpha_k=1/k$ & Time, $\alpha_k=1/k$ & Iter, $\alpha_k=1/\sqrt{k}$ & Time, $\alpha_k=1/\sqrt{k}$ \\
\hline
Halpern A3PM & 10 & 10 & 414 & \textbf{0.02} & 154643 & 5.99 \\
Halpern A3PM $\parallel$ & 10 & 10 & 414 & 0.19 & 154643 & 6.34 \\
Halpern 3PM & 10 & 10 & 339 & 0.94 & 89954 & 210.67 \\
Halpern 3PM $\parallel$ & 10 & 10 & 339 & 0.70 & 89954 & 252.02 \\
Halpern MAP & 10 & 10 & 14671627 & 600 & 1816500 & 600 \\
Halpern Cimmino & 10 & 10 & 14953274 & 600 & 310060 & 600 \\
Halpern Cimmino $\parallel$ & 10 & 10 & 9047794 & 600 & 304512 & 600 \\
Halpern SCCRM & 10 & 10 & 9200528 & 600 & 1984184 & 600 \\
Halpern CRM & 10 & 10 & 1969758 & 600 & 181763 & 600 \\
Dykstra & 10 & 10 & \textbf{26} & 0.08 & \textbf{26} & \textbf{0.05} \\
\hline \hline
Halpern A3PM & 20 & 10 & 513 & \textbf{0.04} & 238960 & 17.03 \\
Halpern A3PM $\parallel$ & 20 & 10 & 513 & \textbf{0.04} & 238960 & 14.19 \\
Halpern 3PM & 20 & 10 & 819 & 2.85 & 211029 & 600 \\
Halpern 3PM $\parallel$ & 20 & 10 & 819 & 1.86 & 258396 & 600 \\
Halpern MAP & 20 & 10 & 7788635 & 600 & 499000 & 600 \\
Halpern Cimmino & 20 & 10 & 6884106 & 600 & 210342 & 600 \\
Halpern Cimmino $\parallel$ & 20 & 10 & 5505978 & 600 & 340995 & 600 \\
Halpern SCCRM & 20 & 10 & 7614891 & 600 & 2234219 & 600 \\
Halpern CRM & 20 & 10 & 802685 & 600 & 166890 & 600 \\
Dykstra & 20 & 10 & \textbf{16} & \textbf{0.04} & \textbf{16} & \textbf{0.04} \\
\hline \hline
Halpern A3PM & 20 & 20 & 588 & 0.12 & 277704 & 27.09 \\
Halpern A3PM $\parallel$ & 20 & 20 & 588 & \textbf{0.10} & 277704 & 25.98 \\
Halpern 3PM & 20 & 20 & 209 & 1.30 & 46970 & 210.47 \\
Halpern 3PM $\parallel$ & 20 & 20 & 209 & 1.02 & 46970 & 251.21 \\
Halpern MAP & 20 & 20 & 6156368 & 600 & 553067 & 600 \\
Halpern Cimmino & 20 & 20 & 5387737 & 600 & 103039 & 600 \\
Halpern Cimmino $\parallel$ & 20 & 20 & 4634747 & 600 & 120520 & 600 \\
Halpern SCCRM & 20 & 20 & 6065947 & 600 & 1815816 & 600 \\
Halpern CRM & 20 & 20 & 464582 & 600 & 79021 & 600 \\
Dykstra & 20 & 20 & \textbf{19} & 0.14 & \textbf{19} & \textbf{0.10} \\
\hline \hline
Halpern A3PM & 20 & 100 & 201 & \textbf{0.34} & 40743 & 42.04 \\
Halpern A3PM $\parallel$ & 20 & 100 & 201 & \textbf{0.34} & 40743 & 33.02 \\
Halpern 3PM & 20 & 100 & \textbf{16} & 0.50 & \textbf{227} & \textbf{4.49} \\
Halpern 3PM $\parallel$ & 20 & 100 & \textbf{16} & 0.50 & \textbf{227} & 7.25 \\
Halpern MAP & 20 & 100 & 441040 & 600 & 187664 & 600 \\
Halpern Cimmino & 20 & 100 & 370772 & 600 & 36008 & 600 \\
Halpern Cimmino $\parallel$ & 20 & 100 & 392097 & 600 & 22686 & 600 \\
Halpern SCCRM & 20 & 100 & 893728 & 600 & 414629 & 600 \\
Halpern CRM & 20 & 100 & 52197 & 600 & 28291 & 600 \\
Dykstra & 20 & 100 & 26699 & 600 & 40965 & 600 \\
\hline
\end{tabular}
}
\label{tableell2}
\end{table}

\begin{table} \centering
\caption{Comparison of iterations and CPU time (s) for Dykstra's algorithm and Halpern iterations driven by various operators: A3PM ($T=T_{\mathrm{A3PM}}$), MAP ($T_{\mathrm{MAP}}$), Cimmino ($T_{\mathrm{Cim}}$), SCCRM ($T_{\mathrm{SCCRM}}$), CRM ($T_{\mathrm{CRM}}$), and 3PM ($T_{\mathrm{3PM}}$). Parallel implementations are denoted by $\parallel$. Parameters: $\varepsilon = 10^{-2}$, $\alpha_k=1/k$. Intersection of polyhedra.}
\vspace{5pt}
\scalebox{0.99}{\begin{tabular}{|c|c|c|c|c|c|}
\hline
Method & $m$ & $n$ & $k$ & Iter & Time (s) \\
\hline
Halpern A3PM & 10 & 10 & 20 & 3948 & 0.49 \\
Halpern A3PM $\parallel$ & 10 & 10 & 20 & 3948 & \textbf{0.47} \\
Halpern 3PM & 10 & 10 & 20 & 840 & 6.01 \\
Halpern 3PM $\parallel$ & 10 & 10 & 20 & 840 & 2.74 \\
Halpern MAP & 10 & 10 & 20 & 805 & 3.82 \\
Halpern Cimmino & 10 & 10 & 20 & 13806 & 75.60 \\
Halpern Cimmino $\parallel$ & 10 & 10 & 20 & 13806 & 26.91 \\
Halpern SCCRM & 10 & 10 & 20 & 3757 & 4.33 \\
Halpern CRM & 10 & 10 & 20 & 26187 & 259.45 \\
Dykstra & 10 & 10 & 20 & \textbf{162} & 1.57 \\
\hline
\hline
Halpern A3PM & 10 & 10 & 10 & 4552 & \textbf{0.12} \\
Halpern A3PM $\parallel$ & 10 & 10 & 10 & 4552 & 0.19 \\
Halpern 3PM & 10 & 10 & 10 & 698 & 3.65 \\
Halpern 3PM $\parallel$ & 10 & 10 & 10 & 698 & 1.47 \\
Halpern MAP & 10 & 10 & 10 & 641 & 2.77 \\
Halpern Cimmino & 10 & 10 & 10 & 5677 & 22.04 \\
Halpern Cimmino $\parallel$ & 10 & 10 & 10 & 5677 & 6.63 \\
Halpern SCCRM & 10 & 10 & 10 & 2128 & 3.97 \\
Halpern CRM & 10 & 10 & 10 & 22776 & 120.39 \\
Dykstra & 10 & 10 & 10 & \textbf{175} & 0.80 \\
\hline
\hline
Halpern A3PM & 10 & 10 & 5 & 7566 & \textbf{0.11} \\
Halpern A3PM $\parallel$ & 10 & 10 & 5 & 7566 & 0.22 \\
Halpern 3PM & 10 & 10 & 5 & 804 & 2.12 \\
Halpern 3PM $\parallel$ & 10 & 10 & 5 & 804 & 1.41 \\
Halpern MAP & 10 & 10 & 5 & 949 & 1.46 \\
Halpern Cimmino & 10 & 10 & 5 & 4670 & 10.11 \\
Halpern Cimmino $\parallel$ & 10 & 10 & 5 & 4670 & 4.72 \\
Halpern SCCRM & 10 & 10 & 5 & 1155 & 2.31 \\
Halpern CRM & 10 & 10 & 5 & 2886 & 10.75 \\
Dykstra & 10 & 10 & 5 & \textbf{133} & 0.30 \\
\hline
\hline
Halpern A3PM & 10 & 10 & 3 & 495411 & 6.09 \\
Halpern A3PM $\parallel$ & 10 & 10 & 3 & 495411 & 13.54 \\
Halpern 3PM & 10 & 10 & 3 & 10739 & 20.54 \\
Halpern 3PM $\parallel$ & 10 & 10 & 3 & 10739 & 18.12 \\
Halpern MAP & 10 & 10 & 3 & 12645 & 18.98 \\
Halpern Cimmino & 10 & 10 & 3 & 81347 & 119.06 \\
Halpern Cimmino $\parallel$ & 10 & 10 & 3 & 81347 & 94.26 \\
Halpern SCCRM & 10 & 10 & 3 & 5578 & 14.25 \\
Halpern CRM & 10 & 10 & 3 & 29541 & 76.71 \\
Dykstra & 10 & 10 & 3 & \textbf{499} & \textbf{0.73} \\
\hline
\end{tabular}
}
\label{tableresp}
\end{table}

\begin{table} \centering
\caption{Comparison of iterations and CPU time (s) for Dykstra's algorithm and Halpern iterations driven by various operators: A3PM ($T=T_{\mathrm{A3PM}}$), MAP ($T_{\mathrm{MAP}}$), Cimmino ($T_{\mathrm{Cim}}$), SCCRM ($T_{\mathrm{SCCRM}}$), CRM ($T_{\mathrm{CRM}}$), and 3PM ($T_{\mathrm{3PM}}$). Parallel implementations are denoted by $\parallel$. Parameters: $\varepsilon = 10^{-2}$, $\alpha_k=1/k$. Intersection of polyhedra.}
\vspace{5pt}
\begin{tabular}{|c|c|c|c|c|c|}
\hline
Method & $m$ & $n$ & $k$ & Iter & Time (s) \\
\hline
Halpern A3PM & 20 & 20 & 20 & 8411 & \textbf{0.71} \\
Halpern A3PM $\parallel$ & 20 & 20 & 20 & 8411 & 0.78 \\
Halpern 3PM & 20 & 20 & 20 & 1164 & 20.22 \\
Halpern 3PM $\parallel$ & 20 & 20 & 20 & 1164 & 6.28 \\
Halpern MAP & 20 & 20 & 20 & 1147 & 16.19 \\
Halpern Cimmino & 20 & 20 & 20 & 20554 & 314.66 \\
Halpern Cimmino $\parallel$ & 20 & 20 & 20 & 20554 & 89.37 \\
Halpern SCCRM & 20 & 20 & 20 & 5948 & 23.36 \\
Halpern CRM & 20 & 20 & 20 & 21173 & 600 \\
Dykstra & 20 & 20 & 20 & \textbf{220} & 3.65 \\
\hline
\hline
Halpern A3PM & 100 & 100 & 20 & 119206 & 103.93 \\
Halpern A3PM $\parallel$ & 100 & 100 & 20 & 119206 & 96.68 \\
Halpern 3PM & 100 & 100 & 20 & 2625 & 600 \\
Halpern 3PM $\parallel$ & 100 & 100 & 20 & 3425 & 296.96 \\
Halpern MAP & 100 & 100 & 20 & 2709 & 545.56 \\
Halpern Cimmino & 100 & 100 & 20 & 2984 & 600 \\
Halpern Cimmino $\parallel$ & 100 & 100 & 20 & 7781 & 600 \\
Halpern SCCRM & 100 & 100 & 20 & 11957 & 600 \\
Halpern CRM & 100 & 100 & 20 & 1475 & 600 \\
Dykstra & 100 & 100 & 20 & \textbf{383} & \textbf{75.35} \\
\hline
\end{tabular}
\label{tablesresp2}
\end{table}

\providecommand{\bysame}{\leavevmode\hbox to3em{\hrulefill}\thinspace}
\providecommand{\MR}{\relax\ifhmode\unskip\space\fi MR }
% \MRhref is called by the amsart/book/proc definition of \MR.
\providecommand{\MRhref}[2]{%
  \href{http://www.ams.org/mathscinet-getitem?mr=#1}{#2}
}
\providecommand{\href}[2]{#2}

\end{document}